\documentclass[a4paper,10pt]{amsart}
\usepackage[utf8]{inputenc}
\usepackage{hyperref}
\usepackage{relsize}
\usepackage{amssymb,amsmath,mathtools,amsthm}
\usepackage{dsfont}
\usepackage[all]{xy}

\usepackage[lmargin=1.1in,rmargin=1.1in,tmargin=1.1in,bmargin=1.1in]{geometry}

\newtheorem{theorem}{Theorem}[section]
\newtheorem{thmx}{Theorem}

\newtheorem{definition}[theorem]{Definition}
\newtheorem{convention}[theorem]{Convention}
\newtheorem{assumption}[theorem]{Assumption}
\newtheorem{example}[theorem]{Example}
\newtheorem{proposition}[theorem]{Proposition}
\newtheorem{lemma}[theorem]{Lemma}

\newtheorem{conjecture}[theorem]{Conjecture}
\newtheorem{corollary}[theorem]{Corollary}

\theoremstyle{remark}
\newtheorem{remark}[theorem]{Remark}

\usepackage{bbm}
\renewcommand{\AA}{\mathbb{A}}

\newcommand{\CC}{\mathbb{C}}

\newcommand{\DD}{\mathbb{D}}

\newcommand{\lfb}{[\![}
\newcommand{\rfb}{]\!]}

\newcommand{\GG}{\mathbb{G}}
\DeclareMathOperator{\dimv}{\underline{dim}}

\DeclareMathOperator{\Hodge}{H}
\DeclareMathOperator{\stable}{st}
\DeclareMathOperator{\Sp}{sp}
\DeclareMathOperator{\ana}{an}
\DeclareMathOperator{\Wt}{W}
\DeclareMathOperator{\Torsion}{T}
\DeclareMathOperator{\sstable}{sst}
\DeclareMathOperator{\Tfree}{F}
\DeclareMathOperator{\Perv}{Perv}
\DeclareMathOperator{\colim}{colim}
\DeclareMathOperator{\framed}{fr}
\DeclareMathOperator{\sframed}{sfr}
\DeclareMathOperator{\Quot}{Quot}

\DeclareMathOperator{\ncHilb}{ncHilb}
\DeclareMathOperator{\cyc}{cyc}
\DeclareMathOperator{\Chow}{Chow}
\DeclareMathOperator{\image}{image}
\DeclareMathOperator{\rmod}{-mod}
\DeclareMathOperator{\reduced}{red}
\DeclareMathOperator{\Def}{\mathcal{D}ef}

\DeclareMathOperator{\BPS}{BPS}
\DeclareMathOperator{\GV}{GV}
\DeclareMathOperator{\BPSs}{\mathcal{BPS}}
\DeclareMathOperator{\KK}{K_0}
\newcommand{\LL}{\mathbb{L}}

\newcommand{\PP}{\mathbb{P}}

\newcommand{\QQ}{\mathbb{Q}}
\newcommand{\QQQ}{\underline{\QQ}}

\newcommand{\ZZ}{\mathbb{Z}}
\newcommand{\Mst}{\mathfrak{M}}

\newcommand{\Msp}{\mathcal{M}}

\newcommand{\OO}{\mathcal{O}}

\newcommand{\phim}[1]{\phi^{\mon}_{#1}}

\DeclareMathOperator{\sst}{-ss}

\DeclareMathOperator{\nilp}{nilp}

\DeclareMathOperator{\ch}{ch}
\DeclareMathOperator{\Art}{\mathbf{Art}}
\DeclareMathOperator{\coker}{coker}

\DeclareMathOperator{\crit}{crit}

\DeclareMathOperator{\Hom}{Hom}

\DeclareMathOperator{\mon}{mon}
\DeclareMathOperator{\End}{End}
\DeclareMathOperator{\Grass}{Grass}
\DeclareMathOperator{\wt}{\chi_{wt}}
\DeclareMathOperator{\wtm}{\chi^{mon}_{wt}}
\DeclareMathOperator{\wth}{\chi^{mon}_{hsp}}
\DeclareMathOperator{\Ext}{Ext}
\DeclareMathOperator{\Exp}{Exp}

\DeclareMathOperator{\MMHM}{MMHM}
\DeclareMathOperator{\MHS}{MHS}

\DeclareMathOperator{\Gr}{Gr}
\DeclareMathOperator{\Vect}{Vect}

\DeclareMathOperator{\rat}{\mathbf{rat}}
\DeclareMathOperator{\lmod}{-mod}

\DeclareMathOperator{\Rep}{Rep}

\DeclareMathOperator{\con}{con}

\DeclareMathOperator{\Coh}{Coh}
\DeclareMathOperator{\QCoh}{QCoh}
\DeclareMathOperator{\cpct}{cpct}
\DeclareMathOperator{\codim}{codim}
\DeclareMathOperator{\Per}{Per}
\DeclareMathOperator{\MHM}{MHM}
\DeclareMathOperator{\IC}{IC}

\DeclareMathOperator{\Sym}{Sym}

\DeclareMathOperator{\red}{red}
\DeclareMathOperator{\Var}{Var}

\DeclareMathOperator{\Spec}{Spec}
\DeclareMathOperator{\Gl}{GL}
\DeclareMathOperator{\MMHS}{MMHS}

\DeclareMathOperator{\id}{id}

\DeclareMathOperator{\Mod}{mod}
\DeclareMathOperator{\Tr}{Tr}

\DeclareMathOperator{\pt}{pt}

\DeclareMathOperator{\rk}{rk}
\DeclareMathOperator{\forg}{\mathbf{forg}}
\DeclareMathOperator{\rforg}{\mathbf{rforg}}

\DeclareMathOperator{\vir}{vir}

\DeclareMathOperator{\Ho}{\mathcal{H}}
\DeclareMathOperator{\HO}{H}

\newcommand{\Db}{\mathcal{D}^{b}}

\title[Refined invariants of flopping curves and finite-dimensional Jacobi algebras]{Refined invariants of flopping curves and finite-dimensional Jacobi algebras}
\author{Ben Davison}

\begin{document}

\begin{abstract}
We define and study refined Gopakumar--Vafa invariants of contractible curves in complex algebraic 3-folds, alongside the cohomological Donaldson--Thomas theory of finite-dimensional Jacobi algebras.  These Gopakumar--Vafa invariants can be constructed one of two ways: as cohomological BPS invariants of contraction algebras controlling the deformation theory of these curves, as defined by Donovan and Wemyss, or by feeding the moduli spaces that Katz used to define genus zero Gopakumar--Vafa invariants into the machinery developed by Joyce et al.  The conjecture that the two definitions give isomorphic results is a special case of a kind of categorified version of the strong rationality conjecture due to Pandharipande and Thomas, that we discuss and propose a means of proving.  We prove the positivity of the cohomological/refined BPS invariants of all finite-dimensional Jacobi algebras.  This result supports this strengthening of the strong rationality conjecture, as well as the conjecture of Brown and Wemyss stating that all finite-dimensional Jacobi algebras for appropriate symmetric quivers are isomorphic to contraction algebras.
\end{abstract}

\maketitle

\section{Introduction}
\subsection{Refined Gopakumar--Vafa invariants}
Let $f\colon X\rightarrow Y$ be a threefold flopping contraction, by which we mean that $X$ is a smooth quasi-projective 3-dimensional variety over $\CC$, $Y$ is Gorenstein, and $f$ is birational, with exactly one exceptional fibre, isomorphic to a rational curve $C$.  Associated to this seemingly innocuous setup is a great deal of rich geometry, captured in part by various enumerative invariants, amongst them the Gopakumar--Vafa invariants $n_{C,1}, n_{C,2},\ldots$.  The \textit{length} \cite[Page 95--96]{CKM88} $l(C)$ of $C$ can be defined to be the length of the structure sheaf of $C$ at the generic point of $C_{\reduced}$, and $l(C)\leq 6$ \cite[Sec.1]{KaMo92}.  For $r$ greater than $l(C)$ we have $n_{C,r}=0$, so that the invariants $n_{C,r}$ for $r\in\mathbb{N}$ provide an essentially a finite list of enumerative invariants for the curve.  

The above Gopakumar--Vafa invariants turn out to be instances of Donaldson--Thomas (DT) invariants, introduced by Richard Thomas as virtual counts of ideal sheaves in Calabi--Yau threefolds \cite{RTthesis, Casson}.  The study of the DT invariants of a (noncompact) Calabi--Yau threefold $X$ via virtual counts of moduli of $A$-modules, where $A$ is a Jacobi algebra derived equivalent to $X$, was pioneered by Szendr\H{o}i in the case of the conifold \cite{Conifold}.  One of the attractive features of this ``noncommutative'' approach to DT theory is that it lends itself very easily to refinement; the virtual Euler characteristics appearing in the theory are equal to the alternating sums of the dimensions of the hypercohomology of vanishing cycle sheaves, and a natural ``categorification'' of the theory is provided by studying the hypercohomology itself.  Moreover this hypercohomology carries a mixed Hodge structure, so that a natural $q$-refinement of noncommutative DT theory is provided by the weight polynomial of these mixed Hodge structures.

Instead of studying the Gopakumar--Vafa invariants via the DT theory of an algebra $A$ that is derived equivalent to $X$, it turns out to be possible to access them via a much smaller Jacobi algebra (i.e. a finite-dimensional one), namely the contraction algebra of Donovan and Wemyss \cite{DW16}.  For instance in the case of the conifold (i.e. the Atiyah flop) this corresponds to replacing the infinite-dimensional Jacobi algebra appearing in \cite{Conifold} with the ground field $\mathbb{C}$.  

In this paper we study the noncommutative Donaldson--Thomas theory of finite-dimensional Jacobi algebras in general, defining a sequence of cohomologically graded Gopakumar--Vafa invariants $\GV_{C,1,0},\GV_{C,2,0} \ldots\in \mathrm{Ob}(\MMHS)$ in the category of monodromic mixed Hodge structures in terms of BPS cohomology of the contraction algebra associated to the flopping curve $C$ (see Definition \ref{GV_co_def})\footnote{In this paper, we reserve the terminology ``Gopakumar--Vafa invariant'' for invariants of moduli of coherent sheaves on threefolds, and ``BPS invariants'' for invariants of Jacobi algebras; as we explain in \S \ref{ThmA_proof_sec}, it is due to the very simple nature of the Hilbert--Chow morphism for the moduli spaces of sheaves that we consider that these two types of invariants coincide for categories of modules over contraction algebras.}.
\begin{thmx}
\label{ThmA}
Each $\GV_{C,r,0}$ is concentrated entirely in cohomological degree zero, and is Verdier self-dual.  The sequence $\GV_{C,1,0},\GV_{C,2,0} \ldots$ categorifies the Gopakumar--Vafa invariants, in the sense that there are equalities $\dim(\GV_{C,r,0})=n_{C,r}$.
\end{thmx}
The object $\GV_{C,r,0}$ is a monodromic mixed Hodge structure, and its dimension is a coarse numerical invariant of it.  By taking more sophisticated invariants (e.g. the weight polynomial, or Hodge spectrum) we define refined Gopakumar--Vafa invariants, in the sense of motivic DT theory.  The most refined cohomological invariant one might take is the class $[\GV_{C,r,0}]$ in the Grothendieck ring of monodromic mixed Hodge structures.  Put in terms of this invariant, the main content of the above theorem is that this invariant is \textit{effective}, since $\GV_{C,r,0}$ is a genuine monodromic mixed Hodge structure, as opposed to a cohomologically graded complex of them.

In the case $l(C)=1$ the threefold $X$ is either the Atiyah flop, or one of Reid's pagoda flops.  The refined DT invariants in these two cases were calculated in \cite{MMNS11,DM12} respectively, after which the refined DT theory for all higher length flops seemed to be out of reach.  Since $n_{C,r}=0$ for $r$ greater than the length of $C$, the effectiveness in Theorem \ref{ThmA} implies that the refined Gopakumar--Vafa invariants also vanish for $r>l(C)$, so that one can hope to completely calculate the refined DT theory of any given higher length flopping curve.  In \cite{VG19} this result is used to realise this hope and completely determine the refined DT theory of an infinite family of length 2 flopping curves.  Indeed, it is conjectured that this is the complete list of isomorphism classes of length two flopping curves, so that these results provide the key to completely determining the refined DT theory of such curves.

\subsection{Positivity of BPS invariants}
The cohomological invariants in Theorem \ref{ThmA} are defined as vanishing cycle cohomology of intersection complexes, via the definition of BPS cohomology for Jacobi algebras formulated by myself and Sven Meinhardt in \cite{DaMe15b}.  For this to make sense, we have to attach to the curve $C\subset X$ a Jacobi algebra, for which we take the contraction algebra $A_{\con}$ introduced by Donovan and Wemyss, representing the noncommutative deformation theory of the coherent sheaf $\mathcal{O}_{C_{\reduced}}$.  The agreement of the numerical invariant $\dim(\GV_{C,r,0})$ and the $r$th Gopakumar--Vafa invariant is then obtained by piecing together a string of equalities due to Behrend \cite[Thm.4.18]{Behr09}, Katz \cite[Sec.3.2]{Katz08}, Hua and Toda \cite[Thm.4.4]{HuTo18}.  

This paper was partly motivated by a conjecture of Brown and Wemyss, stating that \textit{all} finite-dimensional Jacobi algebras (for a select list of symmetric quivers) are isomorphic to contraction algebras of flopping curves; see \cite[Sec.1.6]{BW21} for an exact statement of this conjecture, along with an updated discussion of its status.  This conjecture implies that for any such finite-dimensional Jacobi algebra $A$, the numerical BPS invariants $\omega_{A,1},\ldots$ satisfy the same properties as the genus zero Gopakumar--Vafa invariants arising from a flopping curve, for example they should be \textit{positive}.  

Positivity of BPS invariants is in turn one of the guiding questions in Donaldson--Thomas theory, and generally appears to have quite deep consequences.  For instance, for $Q$ a quiver, the Kac polynomials \cite[Sec.1.14]{kac83} counting the number of absolutely indecomposable $\gamma$-dimensional representations over a field of order $q$ are given by $-q^{-1/2}\omega_{A,\gamma}(q^{-1/2})$ for $A$ a certain Jacobi algebra built out of $Q$, where $\omega_{A,\gamma}(q^{1/2})$ are the refined BPS invariants of this Jacobi algebra \cite{Moz11}.  Adding to the intrigue, the Kac positivity conjecture \cite[Conj.2]{kac83} (proved in \cite{HLRV13}) is thus equivalent to the \textit{negativity} of certain refined BPS invariants!  As a special instance of this phenomenon, the numerical (i.e. unrefined) BPS invariants counting zero-dimensional coherent sheaves on $\mathbb{C}^3$ are all $-1$ \cite[Thm.1]{MNOP2}.  Note that the Jacobi algebras $A$ considered in this paragraph are all infinite-dimensional, so that these observations are at least consistent with the Brown--Wemyss conjecture.  

Putting all of this together, the Brown--Wemyss conjecture says something rather bold about the DT theory of Jacobi algebras, and it was the original intention of this project to disprove the conjecture by finding an example of a finite-dimensional Jacobi algebra with a negative BPS invariant.  Our main theorem shows that this hunt was doomed, and provides strong evidence for the Brown--Wemyss conjecture:

\begin{thmx}
\label{ThmB}
Let $A$ be an algebraic or analytic finite-dimensional Jacobi algebra, for a quiver $Q$ with potential $W$.  Then for all dimension vectors $\gamma\in\mathbb{N}^{Q_0}$ the cohomological BPS invariants $\BPS_{A,\gamma}$ are Verdier self-dual and concentrated in cohomological degree zero.  In particular, the numerical BPS invariants $\omega_{A,\gamma}= \chi(\BPS_{A,\gamma})$ are all positive, and vanish only if the respective cohomological invariant vanishes.  All but finitely many of the invariants $\omega_{\gamma}$ do vanish, and there is a weak inequality
\[
\dim(A)\geq\sum_{\gamma\in\mathbb{N}^{Q_0}}\lvert\gamma\lvert^2\omega_{\gamma}
\]
which becomes an equality if $A$ is an analytic Jacobi algebra.
\end{thmx}

The last statement of the theorem mirrors a theorem of Toda concerning Gopakumar--Vafa invariants of flopping curves \cite[Thm.1.1]{Toda15}.

Since calculating $\BPS_{A,r}$ by brute force for $r\geq 4$ (in particular $r=7$) is essentially impossible, Theorem \ref{ThmB} seems to rule out using DT theory to disprove the Brown--Wemyss conjecture.

\subsection{Purity, monodromy, and deformation invariance}
\label{flopping_curves_sec}
We denote by $C_{\red}$ the reduced curve underlying $C$, so $C_{\red}\cong \mathbb{P}^1$.  The $N$\textit{-type} of $C$ is given by the numbers $(a,b)$, where $N_{C_{\reduced}\lvert X}\cong\mathcal{O}_C(a)\oplus\mathcal{O}_C(b)$.  There are only three possible $N$-types: $(-1,-1)$, $(0,-2)$ and $(1,-3)$.  The first two types correspond to $l(C)=1$, the third to $2\leq l(C)\leq 6$ --- see \cite{KaMo92} for details.  

Up to analytic local isomorphism, there is only one flopping curve of type $(-1,-1)$, the Atiyah flop, given by blowing up the singularity $Y=\Spec(\CC[x,y,z,w]/\langle xy-zw\rangle)$ along the sheaf of ideals $(x,z)$.  
\smallbreak

For $N$-type $(0,-2)$ things become a little more interesting, as there are infinitely many possible analytic local isomorphism classes.  We can still explicitly write them down, and they are parametrised, up to isomorphism, by a single integer.  Up to analytic local isomorphism we may write 
\[
Y=\Spec\left(\CC[x,y,z,w]/\langle xy+(z-w^d)(z+w^d)\rangle \right)
\]
and $p$ is obtained by blowing up along the sheaf of ideals $(x,z-w^d)$.  Here $d\geq 2$ is the \textit{width} defined by Reid \cite[Sec.5]{Rei83} -- if we set $d=1$ we recover the Atiyah flop, and we are back in $N$-type $(-1,-1)$.

We recall the ``classical'' definition of the Gopakumar--Vafa invariants for $C$, following \cite[Sec.2]{BrLeKa99}, to which we also refer for detailed references.  By \cite[Sec.1.14]{Rei83}, a generic hyperplane $Y_0$ passing through $0\in Y$ has a du Val singularity at $0$, and the projection from the proper transform $X_0\rightarrow Y_0$ is a partial resolution, through which the minimal resolution of $Y_0$ factors.  From this we obtain two pieces of combinatorial data: the Dynkin type of the du Val singularity, and the node of the associated Dynkin diagram corresponding to the curve in the minimal resolution $Z_0\rightarrow Y_0$ that is not contracted by the map $Z_0\rightarrow X_0$.  

Letting $\mathcal{Y}\rightarrow \Def(Y_0)$ be a semi-universal deformation of the singular surface $Y_0$, and $\mathcal{X}\rightarrow \Def(X_0)$ be a semi-universal deformation of $X_0$, from the defining equation of the surface $Y_0$ we obtain a map $g:\Delta\rightarrow \Def(Y_0)$ from an analytic disc, such that (analytically locally around the isolated singularity) $Y\cong \Delta\times_{\Def(Y_0)} \mathcal{Y}$ and $X\cong \Delta\times_{\Def(X_0)} \mathcal{X}$ defined via a lift $f:\Delta\rightarrow \Def(X_0)$ of $g$.  By deforming the map $f$, we deform the threefold $X$, and taking a suitably generic such deformation, the curve $C$ is replaced by a number of $(-1,-1)$-curves.  The number $n_{C,r}$ is the number of such curves to occur with homology class tending to $r[C]$ as we deform back to $X$.

We can illustrate the failure of refined DT theory to be fully deformation invariant with the aid of the $(0,-2)$ curves considered above.  In this case, the contraction algebra which determines the refined DT invariants is the Jacobi algebra
\[
A_{\con}\cong \mathbb{C}[t]/t^{d}
\]
meaning that the cohomological BPS invariants are given by the vanishing cycle cohomology of the function $t^{d+1}$ on $\mathbb{A}^1$.  This cohomology is easy to calculate: it is given by the cohomology of the nearby fibre, along with its monodromy operation (i.e. the cohomology of $d+1$ distinct points, with the cyclic permutation action), minus the cohomology of a point, with the trivial action (for the central fibre).  If we now deform the function $g$ defining $Y$ by replacing $w^d$ with some generic degree $d$ polynomial $f(w)$, the calculation of the cohomological BPS invariants instead becomes the calculation of the vanishing cycle cohomology of $h(t)$ across all of its critical values, where $(\partial/\partial t)h(t)=f(t)$.  As we (generically) deform away from $h(t)=t^{d+1}$, the critical values disperse away from each other.  The function $h(t)\colon\mathbb{A}^1\rightarrow \mathbb{A}^1$ still forms a branched cover, but now the monodromy around each critical value is a simple transposition, which from the point of view of monodromic mixed Hodge structures is normalized away (and thus trivial); see Section \ref{MMHMsec}.  

So after deforming $Y$ so that it contains $d$ $(-1,-1)$ curves, each curve contributes $1$ to the refined Gopakumar--Vafa invariant, whereas for $Y$ itself, the monodromies link up to form a more complicated permutation matrix, so that if we work at some level of refinement that remembers monodromy, the refined Gopakumar--Vafa invariant jumps for the function $h(t)=t^{d+1}$.  Nonetheless, the underlying vector space of the cohomological Gopakumar--Vafa invariant doesn't jump, and moreover, since the monodromy is built by composing the transpositions occurring for a generic deformation, it is (up to adding back in a summand corresponding to the central fibre) given by a permutation matrix and is thus semisimple.  Semisimplicity for the monodromy transformation is equivalent to purity for the weight filtration on the monodromic mixed Hodge structure $\GV_{C,r,0}$, by definition of the weight filtration on vanishing cycle cohomology \cite[Sec.5.1.6.2, p.956]{Sa88}.

Putting these hints together, one might guess that while the monodromy transformation on cohomological Gopakumar--Vafa invariants jumps as we deform the threefold $Y$, the underlying vector space does not, and nor does the weight filtration.  Since the Gopakumar--Vafa invariants for the Atiyah flop are pure, this last statement is equivalent to purity of cohomological Gopakumar--Vafa invariants.  Theorem \ref{ThmB} shows that the first of these hopes regarding deformation invariance is true, and Conjecture \ref{purityQ} claims that the second is too.

\subsection{Strong rationality}
We start this subsection by returning to Theorem \ref{ThmA}.  It turns out that there are \textit{several} canonical(!) cohomological lifts of the numbers $n_{C,r}$ (another one is closer to Katz's original definition of the genus zero Gopakumar--Vafa invariants and is discussed in Section \ref{alaK}), given by monodromic mixed Hodge structures $\BPS_{C,r,d}$ for all choices of $d$, where now $d$ records the Euler characteristic of the coherent sheaves on $X$ that this BPS cohomology corresponds to.  This is related to the fact that there are two different ways to define the numbers $n_{C,r}$ via DT theory, one by considering sheaves with Euler characteristic one, and one via sheaves with Euler characteristic zero.  

This invariance of BPS invariants under the choice of Euler characteristic zero or one arises from a general conjecture on invariants of coherent sheaves in 3-folds, going back to the work of Pandharipande and Thomas in \cite{StablePairs}, called the strong rationality conjecture (see \cite{Tod14} for the passage between the invariance conjecture and the strong rationality conjecture).

We conjecture (Conjecture \ref{SRC}) that invariance under change of Euler characteristic lifts to an isomorphism of BPS cohomology $\BPS_{C,r,d}\cong \BPS_{C,r,d'}$ for all $r,d,d'$.  The mysterious aspect of any incarnation of the strong rationality conjecture is the claim that there should be any connection between these invariants for $d-d'$ not a multiple of $r$ (tensoring with line bundles gives isomorphisms between stacks of semistable modules of class $(r,d)$ and those of class $(r,d+tr)$, with $t\in\mathbb{Z}$).  We finish the paper by discussing an explicit morphism between cohomological BPS invariants, constructed via Hecke operators in the Kontsevich--Soibelman cohomological Hall algebra, which we propose as the required isomorphism:
\[
\BPS_{C,r,d}\cong\BPS_{C,r,d+1}.
\]

\subsection{Acknowledgements}
I would like to thank Michael Wemyss for many stimulating discussions regarding contraction algebras and Gopakumar--Vafa invariants.  In addition, most of the hard evidence for the speculations regarding monodromy in the final section are the result of Gavin Brown and Michael Wemyss generously running through the gigantic list of contraction algebra potentials that they have mined.  In addition, I thank Yukinobu Toda for discussions and answers to questions relating to his work, and also Zheng Hua, Francesco Polizzi and Okke van Garderen for helpful discussions.  I would also like to thank the anonymous referee, for making many suggestions that helped with the exposition, and spotting numerous typos.  Finally I thank Francesca Carocci for pointing out a mistake in Section \ref{DEsec}, in an earlier draft.  

This research was supported by the ERC starter grant ``Categorified Donaldson--Thomas theory'' No. 759967 of the European research council, and a Royal Society university research fellowship

\subsection{Plan of the paper}
Section \ref{DTsec} is a review of cohomological Donaldson--Thomas theory for quivers with potential, providing the background for the positivity theorem (Theorem \ref{ThmB}), which we prove in Section \ref{newSec3}.  In Section \ref{curvesec} we review and complement the known facts about the geometry of flopping curves, with the aim of applying the results of Section \ref{DTsec} and \ref{newSec3} to them.  In Section \ref{finalsec} we put the two theories together, defining cohomological Gopakumar--Vafa invariants and proving Theorem \ref{ThmA}.  We finish with some speculations and conjectures regarding cohomological Gopakumar--Vafa invariants, and finally with the categorified version of the strong rationality conjecture.

\section{Cohomological noncommutative Donaldson--Thomas theory}
\label{DTsec}
\subsection{Monodromic mixed Hodge modules}
\label{MMHMsec}
We review briefly those aspects of the theory of mixed Hodge modules that we will need for this paper.  Given an analytic space $Z$, which we will assume to be separated and reduced\footnote{An unusually high proportion of the moduli spaces we deal with in this paper are nonreduced --- we always consider mixed Hodge modules on the reduced subspace.}, we refer to \cite[Sec.2]{Sai90} for the definition of the category $\MHM(Z)$ of mixed Hodge modules on $Z$, as well as proofs of the assertions below, up to the discussion of monodromic mixed Hodge modules.  

The category $\MHM(Z)$ is an upgrade of $\Perv(Z)$, the subcategory of $\Db_c(Z,\QQ)$ of the derived category of $\mathbb{Q}$-sheaves on $Z$ with analytically constructible cohomology sheaves, in the sense that there are fully faithful functors $\MHM(Z)\rightarrow \Perv(Z)$ for all $Z$, commuting with direct and inverse image functors, the Verdier duality functor $\DD_Z$, and external tensor products.

Let $f\colon Z\rightarrow U$ be a morphism of analytic spaces, which we assume to be the analytification of a morphism of complex algebraic varieties.  In particular, this implies projective compactifiability, as in \cite[Sec.2.18]{Sai90}, which is enough to define the cohomological direct image functors.  So there are functors $\Ho^i \!f_*,\Ho^i \!f_!$ from $\MHM(Z)$ to $\MHM(U)$ \cite[Sec.2.18]{Sai90}.  In addition, there are functors $\Ho^i\!f^*$ and $\Ho^i\!f^!$ from $\MHM(U)$ to $\MHM(Z)$ \cite[Prop.2.19]{Sai90}.  If $f$ is a closed embedding then $f_*=f_!$ is exact and $\Ho^i\!f_*=\Ho^i\!f_!=0$ for $i\neq 0$.  If $f$ is an open embedding then $f^*\cong f^!$ is exact, and $\Ho^i\! f^*=\Ho^i \!f^!=0$ for $i\neq 0$.  
  Finally, there are vanishing and nearby cycle functors $\phi_f$ and $\psi_f$, respectively, associated to a given holomorphic function $f$ on $Z$, which upgrade the corresponding functors on perverse sheaves in the obvious sense.

We recall that $\mathcal{L}\in \MHM(Z)$ comes with an ascending weight filtration $\Wt_{\bullet}\!\mathcal{L}$, and $\mathcal{L}$ is called pure of weight $n$ if $\Gr^{\Wt}_i(\mathcal{L})=0$ for $i\neq n$.  A cohomologically graded mixed Hodge module $\mathcal{L}$ will just be called \textit{pure} if $\Ho^i(\mathcal{L})$ is pure of weight $i$.  We denote by $\MHM(Z,n)$ the category of pure polarizable weight $n$ mixed Hodge modules.  A polarization of a weight $n$ Hodge module is given by an isomorphism $\mathcal{L}\cong \DD_Z \mathcal{L}\otimes\LL^n[-2n]$ satisfying some extra conditions (see \cite[Sec.5]{Sa88}), where $\LL:=\HO_c(\AA^1,\mathbb{Q})$.  The simple objects in the category of weight $n$ polarizable Hodge modules are provided by Saito's intermediate extension complexes $\IC_{Z'}(\mathcal{L})$, where $\mathcal{L}$ is a polarizable weight $n$ variation of Hodge structure on a locally closed analytic subvariety $Z'\subset Z$ (see \cite{Sa88}).  Moreover, the category $\MHM(Z,n)$ is semisimple.

Given two complex analytic spaces $Z_0$ and $Z_1$, there is an external tensor product
\[
\boxtimes \colon \MHM(Z_0)\times\MHM(Z_1)\rightarrow \MHM(Z_0\times Z_1)
\]
taking pairs of objects $\mathcal{F}$ and $\mathcal{G}$ in $\MHM(Z_0,n_0)$ and $\MHM(Z_1,n_1)$, respectively, to objects in $\MHM(Z_0\times Z_1, n_0+n_1)$ -- this is, for example, a consequence of Saito's theorem identifying pure mixed Hodge modules with his intermediate extensions of variations of Hodge structure on locally closed analytic subvarieties \cite[Sec.2.f, Sec.3.21, Thm.3.28]{Sai90}.

For the purposes of Donaldson--Thomas theory, it turns out to be convenient to consider mixed Hodge modules with a monodromy action, recording the usual monodromy operator on sheaves of vanishing cycles.  Fixing $Z$, we denote by $\mathcal{B}_Z\subset \MHM(Z\times\mathbb{A}^1)$ the full subcategory of graded-polarizable mixed Hodge modules $\mathcal{L}$ such that each restriction $(z\times\GG_m\hookrightarrow Z\times\AA^1)^*\mathcal{L}$ has locally constant cohomology sheaves, and we denote by $\mathcal{C}_Z\subset \mathcal{B}_Z$ the full Serre subcategory containing those mixed Hodge modules that are in the essential image of the functor 
\[
(Z\times\mathbb{A}^1\rightarrow Z)^*[1]:\MHM(Z)\rightarrow \mathcal{B}_Z.
\]
We define $\MMHM(Z)=\mathcal{B}_Z/\mathcal{C}_Z$, the Serre quotient category.

Direct image along the closed embedding $Z\times\{0\}\hookrightarrow Z\times\AA^1$ defines a functor $\iota\colon\MHM(Z)\rightarrow \MMHM(Z)$.  There is a monoidal product for monodromic mixed Hodge modules defined by setting, for $\mathcal{F}'$ and $\mathcal{F}''$ monodromic mixed Hodge modules on $Z'$ and $Z''$ respectively:
\[
\mathcal{F}'\boxtimes \mathcal{F}'':= (\id_{Z'\times Z''}\times +)_!(\mathcal{F}'\boxtimes^{\MHM}\mathcal{F}'')
\]
where the product $\boxtimes^{\MHM}$ on the right hand side is the usual exterior product of mixed Hodge modules on $Z'\times \AA^1$ and $Z''\times\AA^1$.  If $Z$ is equipped with finite morphisms $\nu:Z\times Z\rightarrow Z$ and $\pt \rightarrow Z$ making it into a monoid in the category of analytic spaces, then $\MMHM(Z)$ acquires a monoidal structure defined by $\mathcal{F}\boxtimes_{\nu}\mathcal{G}:= \nu_*(\mathcal{F}\boxtimes\mathcal{G})$, which is symmetric if $\nu$ is commutative, via the results of \cite{MMS11}.  We only consider the case in which $\nu$ is commutative in this paper.  We write $\Sym_{\nu}^n(\mathcal{F})$ for the $n$th symmetric power of $\mathcal{F}$ with respect to this symmetric monoidal structure, and define 
\begin{equation}
\label{SymCon}
\Sym_{\nu}(\mathcal{F}):=\bigoplus_{n\geq 0}\Sym^n_{\nu}(\mathcal{F})
\end{equation}
when this infinite direct sum makes sense.  Since we have made $\MMHM(Z)$ into a symmetric monoidal category, the Grothendieck group $\KK(\MMHM(Z))$ acquires the structure of a lambda ring in the usual way.  We define
\[
\Exp([\mathcal{F}])=[\Sym_{\nu}(\mathcal{F})]
\]
when the right hand side is well defined.

We denote by $\MMHS$ the category of monodromic mixed Hodge structures on a point.  If $\mathcal{F}$ is a monodromic mixed Hodge module on a space $Z$, and $\mathcal{L}$ is a monodromic mixed Hodge structure, we denote by $\mathcal{F}\otimes\mathcal{L}$ the external product, considered as a monodromic mixed Hodge module on $Z$.

\begin{remark}
The monodromy operator seems to play a key role in the cohomological Donaldson--Thomas theory of flopping curves.  As we will see, in quasi-homogeneous cases, the cohomological BPS invariants are pure of weight zero.  This means that when we pass from numerical DT theory (involving numbers) to refined DT theory (involving polynomials in $q$) nothing interesting happens: we replace numbers with the corresponding constant polynomials.  

So in order to extract interesting polynomial DT invariants it is not sufficient to just consider the weight filtration, one must also keep track of monodromy.  Conjecture \ref{purityQ} states that this is true even for non quasi-homogeneous potentials.  In the general case, by the definition of the weight filtration on vanishing cycle cohomology \cite[(5.1.6.2)]{Sa88}, Conjecture \ref{purityQ} below is equivalent to the statement that the monodromy actions that we encounter are semisimple.
\end{remark}

Consider the (cohomologically shifted) mixed Hodge module $\QQQ_{\mathbb{G}_m}$ on $\mathbb{G}_m$, and its direct image $\mathcal{L}=t_*\QQQ_{\mathbb{G}_m}$ along the map $t\colon z\mapsto z^2$.  The (shifted) mixed Hodge module $\mathcal{L}$ carries an action of $\mathbb{Z}/2\mathbb{Z}$, and we let $\mathcal{L}'\subset \mathcal{L}$ be the sign-isotypical component: it is a (shifted) variation of mixed Hodge structure of rank 1, with monodromy around 0 acting by multiplication by $-1$.  We consider $\LL=\HO_c(\AA^1,\QQ)$ as a cohomologically graded monodromic mixed Hodge structure via the embedding $\iota\colon \MHS\rightarrow \MMHS$.  It is possible to show that
\begin{equation}
\label{Lsq}
\left((j\colon\GG_m\rightarrow\AA^1)_!\mathcal{L}'\right)^{\otimes 2}\cong \mathbb{L}
\end{equation}
and so we define 
\[
\LL^{1/2}:=j_!\mathcal{L}'.
\]

If $Z$ is an irreducible analytic space with open dense regular subspace $Z^{\mathrm{reg}}\subset Z$ we define
\[
\IC_Z:=\IC_{Z}(\QQQ_{Z^{\mathrm{reg}}})\otimes\LL^{\dim(Z)/2}
\]
the weight zero pure Tate twist of the intersection complex for the constant variation of Hodge structure on the regular locus.  If $Z=\coprod_i Z_i$ is a union of irreducible analytic spaces with open dense regular subspaces, we define $\IC_Z=\bigoplus_i\IC_{Z_i}$.  
\smallbreak
If $Z$ is a smooth connected stack we define
\begin{align*}
\HO(Z,\mathbb{Q})_{\vir}:=&\HO(Z,\mathbb{Q})\otimes \LL^{-\dim(Z)/2}\\
\HO_c(Z,\mathbb{Q})_{\vir}:=&(\HO(Z,\mathbb{Q})\otimes \LL^{-\dim(Z)/2})^*
\end{align*}
Given $f$ a holomorphic function on an analytic space $Z$, we define
\[
\phim{f}:=(Z\times\GG_m)_!\phi_{f/u}(Z\times\GG_m\rightarrow Z)^*:\MHM(Z)\rightarrow \MMHM(Z)
\]
where $u$ is the coordinate on $\GG_m$.  Note that this functor is exact.

If $f\colon U\rightarrow\CC$ is a holomorphic function, $p\colon Z\rightarrow U$ is a projective morphism of algebraic varieties, and $\mathcal{F}\in \MHM(Z)$ is polarizable, then there is an isomorphism \cite[Thm.2.14]{Sai90}
\[
\phim{f}\Ho^i \!p_*\mathcal{F}\cong \Ho^i\!p_*\phim{fp}\mathcal{F}.
\]

Below we will want to consider vanishing cycle cohomology of analytic functions on stacks.  Mindful of the fact that the reader may be somewhat unfamiliar with analytic stacks, we next make concrete the construction of this cohomology.

Let $q\colon Z\rightarrow U$ be a (not necessarily projective) morphism of $G$-equivariant varieties, where $G$ is an affine algebraic group acting trivially on $U$, so that there is an induced morphism $Z/G\rightarrow U$ from the Artin quotient stack.  Let $U^{\Sp}\subset U^{\ana}$ be an analytic subset, which for us will always be the analytification of an algebraic subvariety, and let $f$ be a holomorphic function defined on an open analytic subset $U'\subset U$ with $U^{\Sp}\subset U'$.  Write $Z^{\Sp}=q^{-1}(U^{\Sp})$.  We will make the following assumption.
\begin{assumption}
The subspace $Z^{\Sp}\cap \crit(fq)$ is open inside $\crit(fq)$.
\end{assumption}
The assumption implies that restricting a mixed Hodge module that has support contained in $\crit(fq)$ to $Z^{\Sp}\cap \crit(fq)$ gives a mixed Hodge module (as opposed to a complex).

Fix a cohomological degree $i$, and let $V$ be a $G$-representation, such that there is an open sub $G$-variety $V'\subset V$, acted on freely by $G$, with $\codim_V(V\setminus V')\gg 0$.  We define
\begin{equation}
\label{ZTdef}
\tilde{Z}':=(q^{-1}(U')\times V')/G
\end{equation}
and define $\tilde{Z}^{\Sp}$ similarly.  The $G$-action in (\ref{ZTdef}) is the diagonal one.  Denote by $\tilde{q}\colon \tilde{Z}'\rightarrow U'$ the projection, and by $\tilde{f}\in \Gamma(\tilde{Z}')$ the function induced by $f$.  Define 
\begin{align}\label{lane1}
\Ho^i(q_*((\phim{fq}\IC_{Z'/G})\lvert_{Z^{\Sp}/G}))=&\Ho^i\left(\tilde{q}_*((\phim{\tilde{f}}\QQQ_{\tilde{Z}'})_{\tilde{Z}^{\Sp}})\otimes\LL^{(\dim(G)-\dim(Z))/2}\right)\\
\label{lane2}
\Ho^i(q_!((\phim{fq}\IC_{Z'/G})\lvert_{Z^{\Sp}/G}))=&\Ho^i\left(\tilde{q}_!((\phim{\tilde{f}}\QQQ_{\tilde{Z}'})_{\tilde{Z}^{\Sp}})\otimes\LL^{-\dim(V)+(\dim(G)-\dim(Z))/2}\right).
\end{align}
Then as in \cite[Sec.4.1]{DaMe15b} the monodromic mixed Hodge modules on the right hand side of \eqref{lane1} and \eqref{lane2} depend on the choice of pair $V'\subset V$ only up to canonical isomorphism, and so the MMHMs on the left hand side are well-defined.
\subsection{Refined invariants from monodromic mixed Hodge structures}
\label{specSec}
Given a monodromic mixed Hodge structure $\mathcal{L}$, the universal invariant (in the sense of taking extensions to sums) is of course $[\mathcal{L}]\in\KK(\MMHS)$.  We next discuss a few more manageable refined invariants.  In this paper ``refined invariants'' are taken to be any polynomials recovering a known numerical invariant after setting all variables to 1.

We refer to \cite[Sec.4.4]{COHA} for proofs of statements below regarding weight filtrations and Serre polynomials for $\MMHS$s.    

Let $\mathcal{D}\subset \MMHM(\mathbb{G}_m)$ be the smallest full subcategory of polarizable mixed Hodge modules, closed under extensions, containing all pure variations of Hodge structure --- i.e. those mixed Hodge modules such that the underlying perverse sheaf is locally constant.  Then the natural map $j_!\colon \mathcal{D}\rightarrow \MMHS$ is an equivalence of categories.  Denote by $G\colon \MMHS\rightarrow \mathcal{D}$ an inverse equivalence.  

There is a version of the Serre polynomial for monodromic mixed Hodge structures.  We let $\mathcal{L}$ be a cohomologically graded MMHS, and we define
\[
\wtm(\mathcal{L},q^{1/2}):=\sum_{i,j\in\mathbb{Z}}(-1)^j\rk(G(\Gr^i_{\Wt}(\Ho^j(\mathcal{L}))))q^{i/2}.
\]
The weight polynomial satisfies 
\begin{align}
\wtm(\mathcal{L}\otimes\mathcal{P},q^{1/2})&=\wtm(\mathcal{L},q^{1/2})\wtm(\mathcal{P},q^{1/2})\\
\wtm(\mathcal{L},q^{1/2})&=\wtm(\mathcal{L}',q^{1/2})+\wtm(\mathcal{L}'',q^{1/2}) &\;\;\textrm{if}\;0\rightarrow\mathcal{L}'\rightarrow\mathcal{L}\rightarrow\mathcal{L}''\rightarrow 0\;\textrm{is exact.}\label{wtadd}
\end{align}
\begin{remark}
A cohomologically graded MMHS is called \text{pure} if its $i$th cohomology is pure of weight $i$.  It follows from the definition that if $\mathcal{L}$ is such a pure cohomologically graded MMHS then $\wtm(\mathcal{L},q^{1/2})\in\mathbb{N}[(-q^{1/2})^{\pm 1}]$, and for a pure MMHS (i.e. a complex concentrated in cohomological degree zero) we have $\wtm(\mathcal{L},q^{1/2})\in\mathbb{N}$.
\end{remark}
\begin{example}
The monodromic mixed Hodge structure $\LL^{1/2}$ is pure, concentrated in cohomological dimension one, and we have 
\[
\wtm(\mathbb{L}^{1/2},q^{1/2})=-q^{1/2}.
\]
\end{example}
We define the forgetful functor
\begin{align*}
\forg\colon&\MMHS\rightarrow\MHS\\
&\mathcal{L}\mapsto (1\hookrightarrow \GG_m)^*G(\mathcal{L})[-1].
\end{align*}
This functor does not respect the monoidal structures.  It will sometimes be useful to consider a less refined functor that does.  Firstly, the usual forgetful functor taking a mixed Hodge module to its underlying perverse sheaf induces a forgetful functor
\[
\rat\colon \MMHS\rightarrow \mathcal{B}_{\Perv}/ \langle \mathbb{Q}_{\mathbb{A}^1}[1]\rangle^{\oplus}
\]
where $\mathcal{B}_{\Perv}$ is the category of perverse sheaves on $\mathbb{A}^1$ that are locally constant away from zero, and we have taken the quotient by the Serre subcategory of constant perverse sheaves.  We give this quotient the convolution monoidal product, so that $\rat$ is a symmetric monoidal functor.  Let $x$ be a coordinate on $\mathbb{A}^1$.  Then $\phi_x$, the usual vanishing cycle functor for perverse sheaves, induces a functor
\[
F:\mathcal{B}_{\Perv}/ \langle \mathbb{Q}_{\mathbb{A}^1}[1]\rangle^{\oplus}\rightarrow \Vect
\]
which is a monoidal functor by the Thom--Sebastiani isomorphism \cite{Ma01}, and furthermore a symmetric monoidal functor by \cite[Prop.3.11]{DaMe15b}.  We denote by
\[
\rforg:\MMHS\rightarrow \Vect
\]
the composition of functors $F\circ \rat$, and for $\mathcal{F}$ a monodromic mixed Hodge structure we will refer to $\rforg(\mathcal{F})$ as the underlying vector space of $\mathcal{F}$.

The composition $\forg \circ\iota$ is naturally isomorphic to the identity functor.  For $\mathcal{L}$ a MMHS, we define
\[
\dim(\mathcal{L})=\dim(\forg(\mathcal{L}))=\dim(\rforg(\mathcal{L}))
\]
to be the $\QQ$-dimension of the underlying mixed Hodge structure/vector space.  In particular, this number is always nonnegative.  If $\mathcal{L}$ is a cohomologically graded MMHS we define
\[
\chi(\mathcal{L})=\sum_{i\in\mathbb{Z}}(-1)^i\dim(\Ho^i(\forg(\mathcal{L}))).
\]
Note that $\wtm(\mathcal{L},1)=\chi(\mathcal{L})$.  In particular, if $\mathcal{L}$ is a MMHS written as $j_!\mathcal{L}'[1]$ for some variation of Hodge structure on $\GG_m$, $\dim(\mathcal{L})$ is simply the dimension of the fibre of $\mathcal{L}'$ at any point in $\GG_m$: it is insensitive to the monodromy around $0$.  So if $f$ is a function on a smooth complex manifold, there is an equality
\begin{equation}
\label{forgMon}
\chi(\HO(X,\phi_f\QQQ_X[-1]))=\chi(\HO(X,\phim{f}\QQQ_X)).
\end{equation}
\begin{proposition}
\label{mSprop}
Let $\mathcal{L}$ be a monodromic mixed Hodge structure, and let $\mathcal{P}=\forg(\mathcal{L})$.  Then there is a decomposition of mixed Hodge structures $\mathcal{P}=\mathcal{P}^0\oplus\mathcal{P}^{\neq 0}$ into summands for which the underlying $\pi_1(\GG_m)$-representations are unipotent, or have no unipotent subrepresentation, respectively.  With respect to these there is an equality
\[
\wtm(\mathcal{L},q^{1/2})=\wt(\mathcal{P}^0,q^{1/2})+q^{1/2}\wt(\mathcal{P}^{\neq 0},q^{1/2}).
\]
\end{proposition}
We have used $\wt$ to denote the usual weight polynomial for (non-monodromic) mixed Hodge structures.
\begin{proof}
Since $\mathcal{L}$ is assumed graded-polarizable, as a mixed Hodge module on $\AA^1$, the decomposition of $\mathcal{P}$ according to eigenvalues (which are themselves roots of unity) is due to Borel, see \cite{Sch73} --- this decomposition is induced by the decomposition demonstrated there of $\mathcal{G}=G(\mathcal{L})$ according to generalized eigenvalues of the monodromy transformation.  By additivity (\ref{wtadd}) it is then enough to prove the statement for pure Hodge modules $j_!\mathcal{G}$, with $\mathcal{G}$ a simple polarized variation of Hodge structure on $\mathbb{C}^*$ for which the monodromy of $\mathcal{G}$ either does not have 1 as an eigenvalue, or is the identity.  In the first case $j_!\mathcal{G}\cong j_*\mathcal{G}$ is a pure intersection complex, and simple in $\MMHS$, and the statement follows.  In the second case we instead have $j_!\mathcal{G}\cong (0\hookrightarrow \mathbb{A}^1)_*\mathcal{G}_x[-1]$ in $\MMHS$, where $\mathcal{G}_x$ is the fibre of $\mathcal{G}$ at any point $x\in\mathbb{C}^*$ --- this accounts for the shift of weights.
\end{proof}
\begin{remark}
Let $\mathcal{L}$ be a cohomologically graded mixed Hodge structure.  Then Proposition \ref{mSprop} implies that $\wtm(\iota(\mathcal{L}),q^{1/2})=\wt(\mathcal{L},q^{1/2})$.
\end{remark}

If Conjecture \ref{purityQ} is true, then Theorem \ref{ThmA} implies that the Serre polynomials of the $r$th refined BPS/Gopakumar--Vafa invariants of flopping curves are equal to $n_{C,r}$, considered as constant polynomials.  As such, they are conjectured to hold no new information above the classical invariants.  As such, we consider a more refined realization of monodromic mixed Hodge structures called the Hodge spectrum.  For $\mathcal{L}$ a cohomologically graded MMHS with quasi-unipotent monodromy, with all generalized eigenvalues $d$th roots of unity, as in \cite[Sec.4.3]{KS} we define $\mathcal{P}$ as in Proposition \ref{mSprop}, and set
\begin{align*}
\wth(\mathcal{L},z_1,z_2):=&\sum_{i,j,n\in\mathbb{Z}}(-1)^n\dim\left(\Gr^{\Hodge}_i  \left(\Gr_{\Wt}^{j} \Ho^n(\mathcal{P})\right)^{0}_{\CC}\right)z_1^iz_2^{j-i}\\
&+\sum_{1\leq \alpha\leq d-1}\sum_{i,j,n\in\mathbb{Z}}(-1)^n\dim\left(\Gr^{\Hodge}_i  \left(\Gr_{\Wt}^{j} \Ho^n(\mathcal{P})\right)^\alpha_{\CC}\right)z_1^{i+\frac{\alpha}{d}}z_2^{j-i+\frac{d-\alpha}{d}}.
\end{align*}
The numbers $\{0,\ldots, d-1\}$ in the superscripts keep track of the character of the action of $\mu_d$, the $d$th roots of unity, and so a generalized eigenvalue of the quasi-unipotent monodromy action.  We have tensored with $\CC$ so that this decomposition into generalised eigenspaces makes sense.  Note that the exponents of this polynomial are rational, but we still have $\wth(\mathcal{L},q^{1/2},q^{1/2})=\wtm(\mathcal{L},q^{1/2})$, by Proposition \ref{mSprop}.

We arrange the various rings in which refined invariants live, in order of descending ``refinement'':

\[
\xymatrix{
\mathrm{K}_0^{\hat{\mu}}(\Var/\pt)[[\mathbb{A}]^{\pm 1/2}]\ar[r]^-{\chi_{\MMHS}}&\KK(\MMHS)\ar@/^1.5pc/[rrr]^{\wt}\ar@/_1pc/[rrrd]_{\chi}\ar[r]_-{\wth}&\ZZ[z_1^az_2^b\lvert a+b\in\mathbb{Z}]\ar[rr]_-{z_1,z_2\mapsto q^{1/2}}&&\mathbb{Z}[q^{\pm 1/2}]\ar[d]_-{q^{1/2}\mapsto 1}\\&&&&\mathbb{Z}.
}
\]
On the left we have also included one further ring, the naive Grothendieck ring of $\hat{\mu}$-equivariant motives.  As an Abelian group this is generated by symbols $[X]$, where $X$ carries a $\mu_n$-action for some $n$, and we impose the cut and paste relations, identify equivariant varieties up to isomorphism, and use a product that is analogous to the convolution product on monodromic mixed Hodge structures.  We refer to \cite[Sec.4]{DM11} for the remaining relations in this ring, and a thorough account of its properties.  Here we only define the ring homomorphism $\chi_{\MMHS}$:
\[
\chi_{\MMHS}: [\mu_n\curvearrowright X]\mapsto -[(X\times_{\mu_n}\mathbb{C}^*\xrightarrow{(x,z)\mapsto z^n} \mathbb{A}^1)_!\underline{\mathbb{Q}}_{X\times_{\mu_n}\mathbb{C}^*}]
\]
where $X\times_{\mu_n}\mathbb{C}^*$ is the mapping torus, i.e. the quotient of $X\times\mathbb{C}^*$ by the action $\omega\cdot (x,z)=(\omega\cdot x,z\omega^{-1})$.  The fact that this morphism respects the relations in $\mathrm{K}_0^{\hat{\mu}}(\Var/\pt)[\mathbb{L}^{1/2}]$ is easy to show, and the fact that it respects the product and $\lambda$-ring structures follows from the fact that these structures in both rings are provided by convolution (see \cite{DM11} for details).  Note that if the $\mu_n$-action on $X$ is trivial, then
\begin{align*}
\chi_{\MMHS}([X])=&-[(X\times \mathbb{C}^*\xrightarrow{(x,z)\mapsto z}\mathbb{A}^1)_!\underline{\mathbb{Q}}_{X\times \mathbb{C}^*}]\\
=&[(X\xrightarrow{x\mapsto 0}\mathbb{A}^1)_!\underline{\mathbb{Q}}_X]\\
=&[\iota(\HO(X,\mathbb{Q}))]
\end{align*}
where the second equality follows from the existence of the distinguished triangle
\[
(X\times\mathbb{C}^*\hookrightarrow X\times\mathbb{A}^1)_!\underline{\mathbb{Q}}_{X\times\mathbb{C}^*}\rightarrow \underline{\mathbb{Q}}_{X\times\mathbb{A}^1}\rightarrow (X\times\{0\}\hookrightarrow X\times\mathbb{A}^1)_!\underline{\mathbb{Q}}_X
\]
and the isomorphism $(X\times\mathbb{A}^1\xrightarrow{\pi_{\mathbb{A}^1}}\mathbb{A}^1)_!\underline{\mathbb{Q}}_{X\times\mathbb{A}^1}\cong 0$ in $\Db(\MMHS)$.
\subsection{Quivers and potentials}
Throughout, $Q$ will denote a finite quiver, i.e. a pair of finite sets $Q_1$ and $Q_0$ (the arrows and vertices, respectively), along with two maps $s,t\colon Q_1\rightarrow Q_0$ taking an arrow to its source and target, respectively.  We fix $\mathcal{P}$ to be the set of paths in $Q$, including, by convention, a path $e_i$ of length zero for each $i\in Q_0$, beginning and ending at $i$.  For $p\in \mathcal{P}$, let $\lvert p\lvert$ be the length of $p$.  

For a field $K$, which we give the discrete topology, we denote by $KQ$ the free path algebra of $Q$ over $K$.  This is the algebra having as basis the set $\mathcal{P}$, with multiplication defined on this basis by concatenation, where possible, and zero otherwise.  The vector space $KQ_{\geq N}$ spanned by paths of length at least $N$ is a two sided ideal in $KQ$.  We call a left $KQ$-module $M$ nilpotent if for all $m\in M$ there exists an $N\in\mathbb{N}$ such that $KQ_{\geq N}\cdot m=0$.  We denote by $K\lfb Q\rfb$ the topological algebra obtained by completing with respect to the ideal $KQ_{\geq 1}$.  The inclusion of underlying algebras
\[
KQ\rightarrow K\lfb Q\rfb
\]
induces a fully faithful embedding of categories
\[
K\lfb Q\rfb \rmod \rightarrow KQ\rmod
\]
which induces an equivalence of categories between the category of finite-dimensional continuous $K\lfb Q\rfb$-modules and the category of finite-dimensional nilpotent $KQ$-modules.  For $K$ a normed field (for instance, $\CC$ with its usual norm) we define as in \cite[Sec.2.2]{To18}
\[
K\{Q\}\subset K\lfb Q\rfb
\]
the subalgebra consisting of formal linear combinations $\sum_{p\in \mathcal{P}} a_p p$ such that $\lvert a_p\lvert < C^{\lvert p \lvert}$ for some constant $C\in\mathbb{R}_+$ that is independent of $p$.

Given a subset $S\subset K\lfb Q\rfb$ we denote by $K\lfb Q\rfb/\overline{ S}$ the quotient by the topological closure of $S$ (still with respect to the $KQ_{\geq 1}$-adic topology), and we denote by $\langle S\rangle$ the two-sided ideal generated by $S$.  Let $\prod_{i=1}^n a_{i} \in KQ_{\cyc}:=KQ/[KQ,KQ]$ be a single cyclic path in $KQ$, i.e. each $a_i\in Q_1$.  Then we define
\[
\partial W/\partial a=\sum_{a_i=a}a_{i+1}\cdots a_na_1\cdots a_{i-1}
\]
and extend to maps $\partial/\partial a$ from $KQ_{\cyc}$ and $K\lfb Q\rfb_{\cyc}:=K\lfb Q\rfb/\overline{[KQ,KQ]}$ to $KQ$ and $K\lfb Q\rfb$, respectively, by linearity and continuity.  Let $W\in K\lfb Q\rfb_{\cyc}$, then we define
\[
K\lfb Q,W\rfb:= K\lfb Q\rfb/\overline{\langle \partial W/\partial a \;\lvert\; a\in Q_1\rangle}.
\]
and define $K(Q,W)= KQ/\langle \partial W/\partial a\;\lvert \; a\in Q_1\rangle$ if $W$ is algebraic, i.e. it is in the image of the inclusion $KQ_{\cyc}\rightarrow K\lfb Q\rfb_{\cyc}$.  We say that $W$ is analytic if it is in the image of the inclusion $K\{Q\}_{\cyc}:=K\{Q\}/ (\overline{[KQ,KQ]}\cap K\{Q\})\rightarrow K\lfb Q\rfb_{\cyc}$, and for an analytic potential $W$ we likewise define
\[
K\{Q,W\}:=K\{Q\}/(\overline{\langle\partial W/\partial a \;\lvert\;a\in Q_1\rangle}\cap K\{Q\}).
\]
\begin{definition}
An \textit{algebraic} Jacobi algebra $A$ is an algebra isomorphic to an algebra $K(Q,W)$ as defined above.  An analytic Jacobi algebra is one that is isomorphic to $K\{Q,W\}$ as defined above.
\end{definition}

\subsection{Spaces of semistable quiver representations}

In this subsection we set the ground field to be $\CC$.  A $\CC Q$-module $M$ is defined by the tuple of complex vector spaces $(M_i:=e_i\cdot M)_{i\in Q_0}$ and the linear maps $(a\cdot\colon M_{s(a)}\rightarrow M_{t(a)})_{a\in Q_1}$ between them.  We define $\dimv(M)\in\mathbb{N}^{Q_0}$ to be the tuple of numbers $\left(\dim(M_i)\right)_{i\in Q_0}$.

For $\gamma\in\mathbb{N}^{Q_0}$ we define
\[
\AA_{Q,\gamma}:=\prod_{a\in Q_1}\Hom\left(\CC^{\gamma_{s(a)}},\CC^{\gamma_{t(a)}}\right),\quad\quad\Gl_{\gamma}:=\prod_{i\in Q_0} \Gl_{\gamma_i}(\CC).
\]
Then there is an isomorphism
\begin{equation}
\label{stackIso}
\AA_{Q,\gamma}/\Gl_{\gamma}\rightarrow \Mst_{\gamma}(Q)
\end{equation}
where the target is the stack of $\gamma$-dimensional $\CC Q$-modules.  We write $\pi_{\gamma}\colon\AA_{Q,\gamma}\rightarrow \Mst_{\gamma}(Q)$ for the $G_\gamma$-bundle projection.  Denote by $\Msp_{\gamma}(Q)$ the coarse moduli space of the stack $\Mst_{\gamma}(Q)$.  Then
\[
\Msp_{\gamma}(Q)\cong \Spec(\Gamma(\AA_{Q,\gamma})^{G_{\gamma}})
\]
and for a field extension $K\supset \CC$, the $K$-points of $\Msp_{Q,\gamma}$ correspond to $\gamma$-dimensional semisimple $KQ$-modules.  We denote by $p_{\gamma}\colon \Mst_{\gamma}(Q)\rightarrow \Msp_{\gamma}(Q)$ the affinization map.  The affinization is a coarse moduli space of semisimple representations, and the GIT moduli space associated to the trivial linearization.  

Let $\ncHilb_{\gamma}(Q)$ be the fine moduli scheme of pairs $(M,v\in M)$ such that $M$ is a $\gamma$-dimensional $\CC Q$-module and $v$ generates $M$ under the action of $\CC Q$.  This is constructed as an example of a fine moduli space of stable $\CC Q'$-modules for $Q'$, the framed quiver obtained by adding one vertex $\infty$ to $Q$, one arrow from $\infty$ to each of the original vertices of $Q$, and extending the dimension vector by setting the dimension vector at $\infty$ to be $1$; see \cite{King94} for the general construction of GIT moduli spaces of quiver representations, and \cite[Sec.2.2]{MeRe14} for the particular examples that we consider here.  In particular the forgetful map
\begin{equation}
\label{qdef}
q_{\gamma}\colon \ncHilb_{\gamma}(Q)\rightarrow \Msp_{\gamma}(Q)
\end{equation}
is projective, as there is an identification between the coarse moduli spaces of semisimple $\CC Q$-modules and $\CC Q'$-modules, and so (\ref{qdef}) is a GIT quotient map.  

More generally, we denote by $\Mst^{n\framed}_{\gamma}(Q)$ the moduli stack of pairs 
\begin{equation}
\label{frSt}
\{(M,h\colon\mathbb{C}^n\rightarrow M)\lvert\:M\;\textrm{a}\;\gamma\textrm{-dimensional}\;\CC Q\textrm{-module}, h\in\Hom_{\Vect}(\CC^n,M)\}.
\end{equation}

As in (\ref{stackIso}) there is an isomorphism
\[
\Mst^{n\framed}_{\gamma}(Q)\cong \AA_{Q',(\gamma,n)}/\Gl_{\gamma}.
\]
We define $\Msp^{n\sframed}_{\gamma}\subset \Mst^{n\framed}_{\gamma}$ to be the substack defined by the condition that $\image(h)$ generates $M$ as a $\CC Q$-module.  This substack is in fact a scheme, and the fine moduli space of stable pairs (\ref{frSt}).  Then the forgetful map $q^{n}_{\gamma}\colon \Msp^{n\sframed}_{\gamma}(Q)\rightarrow \Msp_{\gamma}(Q)$ is again a GIT quotient map, and so it is proper.  These maps play a key role in cohomological Donaldson--Thomas theory.  In brief: just as we can approximate the cohomology of the \textit{stack} $\Mst_{\gamma}(Q)$ via the \textit{schemes} $\tilde{Z}$ in Section \ref{MMHMsec}, for fixed $i$ we have isomorphisms
\begin{align}
\label{APM}
\Ho^i\!p_{\gamma,*}\phim{\Tr(W)}\QQQ_{\Mst_{\gamma}(Q)}\cong &\Ho^i\!q^n_{\gamma,*}\phim{\Tr(W)}\QQQ_{\Msp^{n\sframed}_{\gamma}(Q)}\\
\nonumber
\Ho^i\!p_{\gamma,!}\phim{\Tr(W)}\QQQ_{\Mst_{\gamma}(Q)}\cong &\Ho^{i}\!\left(q^n_{\gamma,!}\phim{\Tr(W)}\QQQ_{\Msp^{n\sframed}_{\gamma}(Q)}\otimes \LL^{-n\lvert\gamma\lvert}\right)
\end{align}
for $n\gg 0$.  So in particular the objects on the left hand side of (\ref{APM}) enjoy the properties we expect of direct images along a proper map.  We say that $p_{\gamma}$ is \textit{approximated by proper maps}.  See \cite[Sec.4.1]{DaMe15b} for more details of this notion, and \cite{DaMe15b} more generally for applications of this concept.
\begin{convention}
We continue the convention from \cite{DaMe15b} that whenever a space or morphism is defined with a dimension vector as a subscript, then we omit the subscript to denote the union across all dimension vectors.  So for example
\[
q\colon \ncHilb(Q)\rightarrow \Msp(Q)
\]
is the locally projective morphism from the fine moduli scheme of \textit{all} finite-dimensional cyclically generated $\CC Q$-modules to the coarse moduli space.  The same convention applies to sheaves and mixed Hodge modules: if $\bigoplus_{\gamma}\mathcal{F}_{\gamma}$ is a mixed Hodge module with each $\mathcal{F}_{\gamma}$ supported on some space $\Msp_{\gamma}$, we abbreviate the direct sum to $\mathcal{F}$.
\end{convention}

There is a morphism 
\begin{equation}
\label{smap}
\oplus\colon \Msp_{\gamma'}(Q)\times\Msp_{\gamma''}(Q)\rightarrow \Msp_{\gamma'+\gamma''}(Q)
\end{equation}
which at the level of points takes pairs of semisimple modules to their direct sum.  This is a finite morphism by \cite[Lem.2.1]{MeRe14}.  We define a symmetric monoidal structure on $\MMHM(\Msp(Q))$ by setting
\[
\mathcal{F}'\boxtimes_{\oplus}\mathcal{F}'':=\oplus_*\left(\mathcal{F}'\boxtimes\mathcal{F}''\right).
\]

The closed points of each $\Msp_{\gamma}(Q)$ correspond to semisimple $\gamma$-dimensional $\CC Q$-modules, and so there is a distinguished point given by the module $\bigoplus_i S_i^{\gamma_i}$ --- this is the module $M$ given by setting $M_i=\CC^{\gamma_i}$ for all $i\in Q_0$ and letting all paths of length greater than zero act via the zero map.  We denote this point $0_{\gamma}\in\mathcal{M}_{\gamma}(Q)$.  The following is proved in \cite[Lem.2.15, Sec 2.6]{To18}.
\begin{proposition}
\label{Toda_an_prop}
Let $W\in K\{Q\}_{\cyc}$ be an analytic potential, and let $\gamma$ be a dimension vector.  Then there is an analytic neighbourhood $U_{\gamma}$ of $0_{\gamma}\in\Msp_{\gamma}(Q)$ such that $\Tr(W)$ converges absolutely on $U_{\gamma}$ and on $\pi_{\gamma}^{-1}p_{\gamma}^{-1}(U_{\gamma})$.
\end{proposition}

\subsection{Stacks of nilpotent modules}
Let $A$ be a finitely generated $\CC$-algebra, with $\mathfrak{m}\subset A$ a two-sided ideal.  We do not assume in this subsection that $\cap_{n\in\mathbb{N}}\mathfrak{m}^n=0$.  Let $\hat{A}$ be the completion of $A$ with respect to $\mathfrak{m}$.  Setting $\Mst_n(A)$ to be the stack of $n$-dimensional $A$-modules, we let $\Mst_{n,\nilp}(A)$ be the reduced substack defined by the equations $\Tr(a)=0$ for $a\in \mathfrak{m}$.  A geometric point of $\Mst_n(A)$ corresponding to an $A\otimes_{\CC} K$-module $\rho$ is in $\Mst_{n,\nilp}(A)$ if and only if every element of $\mathfrak{m}\otimes_{\CC}K$ acts nilpotently on $\rho$.

\begin{definition}
Given a commutative $\CC$-algebra $B$, a flat family of $\mathfrak{m}$-nilpotent $A$-modules over $B$ is a finite projective $B$-module with an action of $A/\mathfrak{m}^n$ for some $n\in \mathbb{N}$.
\end{definition}
As in \cite[Rem.4.31]{DMSS12}, if $B$ is finitely generated, a flat family of nilpotent $A$-modules over $B$ is just the same as a flat family of $A$-modules which is nilpotent at every geometric point.  

\begin{proposition}
\label{nilpProp}
\begin{enumerate}
\item
If $\hat{A}$ is finite-dimensional, the stack $\Mst_{n,\nilp}(A)$ is an open and closed substack of $(\Mst_n(A))_{\reduced}$ for every $n$.
\item
Alternatively, assume that $A=\CC\{Q,W\}$ is an analytic Jacobi algebra, $A$ is finite-dimensional, and $\gamma\in\mathbb{N}^{Q_0}$ is a dimension vector.  Then $\Mst_{\gamma,\nilp}(A)$ is the reduction of an open and closed substack of $\crit(\Tr(W))\cap p^{-1}(U)$, for $U\ni 0_{\gamma}$ an analytic neighbourhood on which $\Tr(W)$ converges. 
\end{enumerate}
\end{proposition}
\begin{proof}
We consider only (1), the proof for (2) is the same.  The closedness follows from the definitions, so we consider only the openness condition.  The argument is essentially that of \cite[Prop.4.28]{DMSS12}.

Let $\Mst^{n\framed}_{n}(A)$ be the moduli stack of pairs $(\rho, h\colon \CC^n\rightarrow \rho)$ of a $n$-dimensional $A$-module and a homomorphism from the fixed vector space $\CC^n$, and consider the subscheme(!) $\Msp_{n}^{n \sframed}(A)$ defined by the open condition that the image of $h$ generates $\rho$ as an $A$-module.  Then the forgetful map $\pi\colon\Mst^{n\framed}_{n}(A)\rightarrow \Mst_n(A)$ is the projection of a vector bundle, and each $\pi^{-1}(x)\cap \Msp^{n \sframed}_{n}(A)$ is open and dense in $\pi^{-1}(x)$.  Define 
\[
\mathcal{Q}:=\Msp^{n \sframed}_{n}(A)\times_{\Mst_n(A)}\Mst_{n,\nilp}(A).
\]
Then openness of $\Mst_{n,\nilp}(A)$ in $(\Mst_{n}(A))_{\reduced}$ is equivalent to openness of the subscheme $\mathcal{Q}_{\reduced}\subset \Msp_{n}^{n \sframed}(A)_{\reduced}$.

Define $\Msp^{n\sframed}_{n,\nilp}(A)$ to be the functor taking a commutative $\CC$-algebra $B$ to the set of pairs 
\begin{equation}
\label{pairs}
(\mathcal{J},h)/\textrm{isomorphism of pairs}
\end{equation}
where $\mathcal{J}$ is a $B$-flat family of nilpotent $A$-modules over $B$, and $h\colon B^{\oplus n}\rightarrow \mathcal{J}$ is a morphism of $B$-modules such that the induced map $B^{\oplus n}\otimes A\rightarrow \mathcal{J}$ is surjective. 

Then we claim that there is a natural isomorphism of functors $\Msp^{n \sframed}_{n,\nilp}(A)\rightarrow F$, with $F(B)=\colim \mathcal{Q}_*(B)$ defined as in Lemma \ref{complLemma} via the inclusion $\mathcal{Q}\subset \Msp^{n \sframed}_{n}(A)$.  This follows just as in the proof of \cite[Prop.4.28]{DMSS12}; both functors commute with filtered colimits, and so it is enough to check on finitely generated $\CC$-algebras.  But then both sides are defined by the condition of geometric points being sent to the closed subscheme $\mathcal{Q}$.

Now assume that $\hat{A}$ is finite-dimensional.  Then a pair as in (\ref{pairs}) is just a $B$-point of $\Grass_{\hat{A}\lmod, n}(\hat{A}^{\oplus n})$, so that this scheme represents the functor $\Msp^{n\sframed}_{n,\nilp}(A)$, and we deduce the result from Lemma \ref{complLemma}.

\end{proof}

\begin{lemma}\cite[Lem.4.29]{DMSS12}
\label{complLemma}
Let $Z\subset U$ be an inclusion of a closed subscheme inside a scheme of finite type, and let 
\begin{equation}
\label{nest}
Z=Z_1\subset Z_2\subset\ldots \subset U
\end{equation}
be the chain of inclusions defined by powers of the ideal $\mathcal{I}_{Z/U}$.  Define the functor $F$ from commutative $\CC$-algebras to sets by 
\[
F(B)=\colim Z_*(B).
\]
Then if $F$ is represented by a scheme of finite type, the sequence (\ref{nest}) stabilizes to a subscheme $Z_N$, and $Z_{\reduced}\subset U_{\reduced}$ is the inclusion of a union of connected components.
\end{lemma}

\section{Cohomological BPS invariants for Jacobi algebras}
\label{newSec3}
\subsection{Integrality}
\label{BPSsec}
We will assume throughout this section that $Q$ is a symmetric quiver, in the sense that for every pair of vertices $i,j\in Q_0$ there are as many arrows from $i$ to $j$ as from $j$ to $i$.  Everything up to and including Definition \ref{BPSdef} can be generalised to non-symmetric quivers carrying King stability conditions -- see \cite{MeRe14} and \cite{DaMe4,DaMe15b} for details.

Define $\IC'_{\Msp_{\gamma}(Q)}\in\MMHM(\Msp_{\gamma}(Q))$ by 
\[
\IC'_{\Msp_{\gamma}(Q)}:=\begin{cases} \IC_{\Msp_{\gamma}(Q)}&\textrm{if there is a simple }\gamma\textrm{-dimensional } \mathbb{C}Q\mathrm{-module}\\
0&\textrm{otherwise}.\end{cases}
\]

In \cite[Prop.4.3+Thm.4.6]{MeRe14} Meinhardt and Reineke prove the following 
\begin{theorem}
\label{MRthm}
There are equalities in the Grothendieck ring of algebraic mixed Hodge modules on $\Msp(Q)$:
\begin{align}
[q^n_!\IC_{\Msp^{n\sframed}(Q)}]=&\left[\Sym_{\oplus}\left(\bigoplus_{\gamma\neq 0} \IC'_{\Msp_{\gamma}(Q)}\otimes \HO(\CC\PP^{n\lvert \gamma\lvert-1},\QQ)_{\vir}\right)\right]
\end{align}
if $n$ is even, and
\begin{align}
[p_!\IC_{\Mst(Q)}]=&\left[\Sym_{\oplus}\left(\bigoplus_{\gamma\neq 0} \IC'_{\Msp_{\gamma}(Q)}\otimes \HO_c(\pt/\mathbb{C}^*,\QQ)_{\vir}\right)\right].
\end{align}
\end{theorem}

See the discussion around (\ref{SymCon}) for the definition of $\Sym_{\oplus}$.
\begin{proposition}
\label{MRc1}
There are isomorphisms of analytic mixed Hodge modules on $\Msp(Q)^{\ana}$:
\begin{align}
\label{MRiso}
\Ho \!\left(q^n_!\IC_{\Msp^{n\sframed}(Q)}\right)\cong&\Sym_{\oplus}\left(\bigoplus_{\gamma\neq 0} \IC'_{\Msp_{\gamma}(Q)}\otimes \HO(\CC\PP^{n\lvert \gamma\lvert-1},\QQ)_{\vir}\right)
\end{align}
if $n$ is even, and
\begin{align}
\label{MRisop}
\Ho \!\left(p_!\IC_{\Mst(Q)}\right)\cong&\Sym_{\oplus}\left(\bigoplus_{\gamma\neq 0} \IC'_{\Msp_{\gamma}(Q)}\otimes \HO_c(\pt/\mathbb{C}^*,\QQ)_{\vir}\right).
\end{align}
\end{proposition}
\begin{proof}
The analogous statement for algebraic mixed Hodge modules is proved as in the first part of the proof of \cite[Thm.4.10]{DaMe15b} in the category of algebraic mixed Hodge modules.  The point is to show that the $i$th cohomologically graded piece of both the left hand side and the right hand side are pure, and then use semisimplicity of the category of pure weight $i$ mixed Hodge modules to argue that if they have the same class in the Grothendieck group, they are isomorphic.  The equality is then provided by the Meinhardt--Reineke theorem (Theorem \ref{MRthm}).

We refer to \cite[Sec.4.7]{Hetal} for the required definition of the analytification functor for $\mathcal{D}$-modules.  The salient points are that analytification commutes with taking the de Rham functor (by definition), and with direct image along proper morphisms \cite[Prop.4.7.2]{Hetal}.  These facts, together with the projectivity of $q$ and $\oplus$, and approximation of $p_{\gamma}$ by proper maps $q^n_{\gamma}$, mean that the analytic versions of (\ref{MRiso}) and (\ref{MRisop}) follow by taking the analytification of the algebraic version.
\end{proof}
We recall that a \textit{Serre subcategory} $\mathcal{S}$ of an Abelian category $\mathcal{A}$ is a full subcategory satisfying the condition that for every short exact sequence
\[
0\rightarrow M'\rightarrow M\rightarrow M''\rightarrow 0
\]
in $\mathcal{A}$, $M'$ and $M''$ are objects of $\mathcal{S}$ if and only if $M$ is.

\begin{proposition}
\label{MRc2}
Let $W\in \CC Q_{\cyc}$ be an algebraic potential.  Let $\Msp(Q)_{\mathcal{S}}\subset \Msp(Q)$ be an inclusion of varieties corresponding to the inclusion of a Serre subcategory $\mathcal{S}$, satisfying the condition that for all $M\in \mathcal{S}$, the function $\Tr(W)$ is zero when evaluated on $M$.  For a cohomologically graded mixed Hodge module $\mathcal{F}$ on $\Msp(Q)$ we denote by $\mathcal{F}_{\mathcal{S}}$ the restriction to $\Msp(Q)_{\mathcal{S}}$.  Then there are isomorphisms of algebraic mixed Hodge modules on $\Msp(Q)^{\ana}$:
\begin{align}
\label{MRiso2}
\Ho \!\left(q^n_!\phim{\Tr(W)}\IC_{\Msp^{n\sframed}(Q)}\right)_{\mathcal{S}}=&\Sym_{\oplus}\left(\bigoplus_{\gamma\neq 0} (\phim{\Tr(W)}\IC'_{\Msp_{\gamma}(Q)})_{\mathcal{S}}\otimes \HO\left(\CC\PP^{n\lvert \gamma\lvert-1},\QQ\right)_{\vir}\right)
\end{align}
if $n$ is even, and
\begin{align}
\label{Sint}
\Ho \!\left(p_!\phim{\Tr(W)}\IC_{\Mst(Q)}\right)_{\mathcal{S}}=&\Sym_{\oplus}\left(\bigoplus_{\gamma\neq 0} (\phim{\Tr(W)}\IC'_{\Msp_{\gamma}(Q)})_{\mathcal{S}}\otimes \HO_c\left(\pt/\mathbb{C}^*,\QQ\right)_{\vir}\right).
\end{align}
Alternatively, let $W\in\CC \{Q\}_{\cyc}$ be an analytic potential, such that there exists an open analytic subscheme $U'_\gamma\subset \Msp_{\gamma}(Q)$, on which $\Tr(W)$ defines an absolutely convergent function $f_{\gamma}$ such that the reduced critical locus of $f_\gamma \pi_{\gamma}$ is the subvariety of nilpotent $\mathbb{C}Q$-modules.  Then there are isomorphisms of analytic mixed Hodge modules
\begin{align}
\label{MRiso2nilp}
\Ho \!\left(q^n_!\phim{\Tr(W)}\IC_{\Msp^{n\sframed}(Q)}\lvert_{\nilp}\right)=&\Sym_{\oplus}\left(\bigoplus_{\gamma\neq 0} \phim{\Tr(W)}\IC'_{\Msp_{\gamma}(Q)}\lvert_{\nilp}\otimes \HO(\CC\PP^{n\lvert \gamma\lvert-1},\QQ)_{\vir}\right)
\end{align}
if $n$ is even, and 
\begin{align}
 \label{integralityIso}
\Ho \!\left(p_!\phim{\Tr(W)}\IC_{\Mst(Q)}\lvert_{\nilp}\right)=&\Sym_{\oplus}\left(\bigoplus_{\gamma\neq 0} \phim{\Tr(W)}\IC'_{\Msp_{\gamma}(Q)}\lvert_{\nilp}\otimes \HO_c(\pt/\mathbb{C}^*,\QQ)_{\vir}\right).
\end{align}
\end{proposition}
\begin{proof}
The algebraic statement is proved as \cite[Thm.A]{DaMe15b}, using the commutativity of the vanishing cycles functor with direct image along proper maps, and the Thom--Sebastiani isomorphism.  The condition on $\mathcal{S}$ is needed to apply the Thom--Sebastiani isomorphism.

The statement involving analytic potentials follows from Proposition \ref{MRc1} via the same argument.  Fix a dimension vector $\gamma$, and let $U_\gamma'$ be as in the statement of the theorem, then by (\ref{MRiso}), there is an isomorphism
\begin{align}\label{hiat}
\phim{\Tr(W)}\Ho \!\left(q^n_!\IC_{\Msp^{n\sframed}(Q)}\right)\lvert_{U'_{\gamma}}\cong&\phim{\Tr(W)}\Sym_{\oplus}\left(\bigoplus_{\gamma\neq 0} \IC'_{\Msp_{\gamma}(Q)}\otimes \HO(\CC\PP^{n\lvert \gamma\lvert-1})_{\vir}\right)\lvert_{U'_{\gamma}}\\ 
\nonumber
\cong&\phim{\Tr(W)}\bigoplus_{\underline{\gamma}\in D(\gamma)}\left(\boxtimes_{\oplus,i} \Sym_{\oplus}^{\lambda_i}\left(\IC'_{\Msp_{\gamma_i}(Q)}\otimes \HO(\CC \mathbb{P}^{n\lvert \gamma_i\lvert-1})_{\vir}\right)\right)\lvert_{U'_{\gamma}}
\end{align}
The right hand side decomposes according to decompositions in
\[
D(\gamma):=\{(\gamma^1,\gamma^2,\ldots,\gamma^p,\lambda_1,\ldots,\lambda_p)\lvert\;\gamma^i\in\mathbb{N}^{Q_0},\;\lambda_i\in\mathbb{N},\;\gamma^i\neq \gamma^j \;\textrm{for} \; i\neq j,;\sum_i\lambda_i\gamma^i=\gamma\}.
\]
On the one hand, there is an isomorphism
\[
\phim{\Tr(W)}\Ho \!\left(q^n_!\IC_{\Msp^{n\sframed}(Q)}\right)\lvert_{U'_{\gamma}}\cong \Ho \!\left(q^n_!\phim{\Tr(W)}\IC_{\Msp^{n\sframed}(Q)}\lvert_{(q^n_{\gamma})^{-1}(U'_{\gamma})}\right)
\]
by properness of $q^n_{\gamma}$.  By our assumption on the spaces $U'_{\gamma}$, the mixed Hodge module $\phim{\Tr(W)}\IC_{\Msp^{n\sframed}(Q)}\lvert_{(q^n_{\gamma})^{-1}(U'_{\gamma})}$ is supported on the nilpotent locus.  It follows that the same is true of all summands on the right hand side of (\ref{hiat}).
We fix analytic open subspaces $U''_{\gamma^i}$ such that 
\[
\bigoplus\left(\prod_i\prod_{j\leq \lambda_i}U''_{\gamma^i}\right)\subset U_{\gamma}'
\]
where the direct sum is the map of analytic subspaces defined as in (\ref{smap}).  We shrink each $U''_{\gamma^i}$ to be contained in $U'_{\gamma^i}$ if necessary.  Then via the support conditions on the spaces $U'_{\gamma'}$, and the Thom--Sebastiani isomophism we write the summand corresponding to the decomposition $\underline{\gamma}$ as
\begin{align*}
&\phim{\Tr(W)}\left(\boxtimes_{\oplus,i} \Sym_{\oplus}^{\lambda_i}\left(\IC_{\Msp_{\gamma_i}(Q)}\lvert_{U''_{\gamma^i}}\otimes \HO(\CC \mathbb{P}^{n\lvert \gamma_i\lvert-1})_{\vir}\right)\right)\cong\\
&\left(\boxtimes_{\oplus,i} \Sym_{\oplus}^{\lambda_i}\left(\phim{\Tr(W)}\IC_{\Msp_{\gamma_i}(Q)}\lvert_{U''_{\gamma^i}}\otimes \HO(\CC \mathbb{P}^{n\lvert \gamma_i\lvert-1})_{\vir}\right)\right).
\end{align*}
\end{proof}
\begin{remark}
For algebraic Jacobi algebras, in this paper we will always take $\mathcal{S}$ to be the maximal Serre subcategory of $\mathbb{C}(Q,W)$-modules satisfying the conditions of Proposition \ref{MRc2}.  In other words, $M$ is an object of $\mathcal{S}$ if every subquotient of $M$ satisfies the condition that $\Tr(W)$ is zero when evaluated on $M$.
\end{remark} 
\subsection{Cohomological BPS invariants}

For the uninitiated, we should say something informal about why (\ref{integralityIso}) is a surprising and useful result.   On the one hand, the statement tells us something that we could already guess: that the left hand side is very large.  I.e. on the right hand side we take mixed Hodge modules $\phim{\Tr(W)}\IC'_{\Msp_{\gamma}(Q)}$ for each $\gamma$, restrict to the nilpotent locus, tensor them with $\HO(\pt/\mathbb{C}^*)_{\vir}$ to obtain something that is already infinite dimensional, and then we take the symmetric algebra generated by this infinite-dimensional object.  The surprising thing is that the gigantic complex of mixed Hodge modules $\Ho \!\left(p_!\phim{\Tr(W)}\IC_{\Mst(Q)}\lvert_{\nilp}\right)$ on the left hand side is determined in this elementary way by the much more manageable objects $\phim{\Tr(W)}\IC'_{\Msp_{\gamma}(Q)}$.

We should also say something about why the previous section goes by the name of ``integrality''.  The numerical DT invariants of quivers with potential are calculated by considering Euler characteristics of the hypercohomology of right hand side of \eqref{MRiso2} and are defined by a kind of inclusion-exclusion type formula that involves denominators \cite[Sec.7]{JS08}.  In particular, it is not clear at all that they should be integers.  Unwinding the definitions a little (see \cite[Sec.6.7]{DaMe4}), these numerical invariants are precisely equal to the hypercohomology of $(\phim{\Tr(W)}\IC'_{\Msp_{\gamma}(Q)})_{\mathcal{S}}$.  In particular, integrality of numerical DT invariants follows from \eqref{MRiso2}

These remarks motivate the following definition.
\begin{definition}
\label{BPSdef}
Let $(Q,W)$ be a quiver with algebraic potential.  We define the cohomological BPS invariants for $\CC(Q,W)$ by setting 
\[
\BPS_{\CC(Q,W),\gamma}:=\HO_c(\Msp_{\gamma}(Q)_{\mathcal{S}},\phim{\Tr(W)}\IC'_{\Msp_{\gamma}(Q)})^*,
\]
the dual compactly supported cohomology of the sheaf $\phim{\Tr(W)}\IC'_{\Msp^{\stable}_{\gamma}(Q)}$ restricted to the locus of points representing modules in $\mathcal{S}$.   If $(Q,W)$ is instead a quiver with analytic potential $W\in \CC\{Q\}_{\cyc}$ and finite dimensional Jacobi algebra $\CC\{Q,W\}$ we define
\[
\BPS_{\CC\{Q,W\},\gamma}:=\HO_c(\Msp_{\gamma}(Q),\phim{\Tr(W)}\IC'_{\Msp_{\gamma}(Q)}\lvert_{\nilp})^*.
\]
\end{definition}
For the relation between the above cohomological invariants and the motivic, or refined DT invariants studied in e.g. \cite{KS} see Appendix \ref{CoMo}.
\begin{remark}
\label{altDef}
Assuming that $\CC\{Q,W\}$ is finite-dimensional, by Proposition \ref{nilpProp}(2), the monodromic mixed Hodge module $\phim{\Tr(W)}\IC'_{\Msp_{\gamma}(Q)}\lvert_{\nilp}$ is a direct summand of $\phim{\Tr(W)}\IC_{U}$, for $U\subset  \Msp_{\gamma}(Q)$ an analytic open subspace on which $\Tr(W)$ converges.  By Verdier self duality of the intersection complex and the vanishing cycles functor, there is an isomorphism
\[
\BPS_{\CC\{Q,W\},\gamma}\cong\HO(\Msp_{\gamma}(Q),\phim{\Tr(W)}\IC'_{\Msp_{\gamma}(Q)}\lvert_{\nilp}).
\]
\end{remark}

If we assume that all $\CC(Q,W)$-modules are in $\mathcal{S}$, then as in Remark \ref{altDef}, we deduce that there is an isomorphism
\[
\BPS_{\CC(Q,W),\gamma}\cong\HO(\Msp_{\gamma}(Q),\phim{\Tr(W)}\IC'_{\Msp_{\gamma}(Q)}).
\]

\begin{proposition}
\label{simplesProp}
Let $\CC(Q,W)$ be a finite-dimensional algebraic Jacobi algebra.  Then the support of $\phim{\Tr(W)}\IC'_{\Msp_{\gamma}(Q)}$ (as a perverse sheaf) is zero-dimensional, and so it is also a constructible sheaf.
\end{proposition}
\begin{proof}
We first claim that the left hand side of (\ref{MRiso2}) is supported on points corresponding to semisimple $\CC(Q,W)$-modules.  Since $\Msp^{n\sframed}(Q)$ is smooth, $\phim{\Tr(W)}\IC_{\Msp^{n\sframed}(Q)}$ is supported on the critical locus of $\Tr(W)$, which is equal to the space of stable framed $\mathbb{C}(Q,W)$-modules. The morphism $q$ takes a stable framed $\mathbb{C}(Q,W)$-module to the semisimplification of the underlying $\mathbb{C}(Q,W)$-module, so the right hand side is supported on the space of semisimple $\mathbb{C}(Q,W)$-modules, as claimed.

So in particular, $\phim{\Tr(W)}\IC_{\Msp_{\gamma}(Q)}$ is supported on points corresponding to semisimple $\CC(Q,W)$-modules.  For a given dimension vector $\gamma$, there are only finitely many of these, as there are finitely many simple $\CC(Q,W)$-modules (each one must occur in the Jordan--Holder filtration of $\CC(Q,W)$).
\end{proof}

We finish this subsection with some remarks about the BPS Lie algebra, defined for any algebraic Jacobi algebra in \cite{DaMe15b}.  We do not give full definitions, since it will turn out that this Lie algebra is rather trivial for finite-dimensional Jacobi algebras --- see Corollary \ref{ZB}, though see also Conjecture \ref{SRC} for a reason for not totally disregarding the cohomological Hall algebra structure that induces the Lie bracket on BPS cohomology.

In \cite{COHA} Kontsevich and Soibelman defined the cohomological Hall algebra $\mathcal{H}_{Q,W}$ associated to an (algebraic) quiver with potential.  In \cite{DaMe15b} it was shown that there is a perverse filtration on $\mathcal{H}_{Q,W}$, and in particular that $\BPS_{\CC(Q,W)}\otimes\LL^{1/2}$ is the first perverse piece of this filtration, and is preserved by the commutator bracket on $\mathcal{H}_{Q,W}$.  As such, the \textit{BPS Lie algebra} is defined to be
\begin{equation}
\label{BPSalgdef}
\mathfrak{g}_{\CC(Q,W)}^+:=\bigoplus_{\gamma>0}\BPS_{\CC(Q,W),\gamma}\otimes\LL^{1/2}.
\end{equation}
For $\CC\{Q,W\}$ a finite-dimensional analytic Jacobi algebra we can define the Lie algebra $\mathfrak{g}_{\CC\{ Q,W\}}$ the same way; the proof that it is a Lie algebra goes exactly the same way as in the algebraic case.
\subsection{Proof of Theorem \ref{ThmB}}
We can now prove our main theorem on Jacobi algebras.
\begin{theorem}
\label{Jacthm}
Let $A=\CC\{Q,W\}$ be an analytic Jacobi algebra, with $Q$ a symmetric quiver, and assume that $A$ is finite dimensional.  Then
\begin{enumerate}
\item
The refined BPS invariants $\omega_{A,\gamma}(z_1,z_2)$ and $\omega_{A,\gamma}(q^{1/2})$ have only nonnegative coefficients.
\item
For $\gamma\in\mathbb{N}^{Q_0}$ the BPS cohomology $\BPS_{A,\gamma}\in\Db(\MMHS)$ is concentrated entirely in cohomological degree zero, and is self-dual.  There is thus an equality $\omega_{A,\gamma}=\dim(\BPS_{A,\gamma})$.
\item
For all but finitely many $\gamma\in\mathbb{N}^{Q_0}$ the BPS invariant $\omega_{A,\gamma}$ vanishes.
\item
There is an equality $\dim(A)=\sum_{\gamma} \dim(\BPS_{A,\gamma}) \lvert \gamma\lvert^2$.
\end{enumerate}
Alternatively, let $A=\CC(Q,W)$ be a finite-dimensional algebraic Jacobi algebra, for a symmetric quiver $Q$.  Then all but the fourth assertions are true of the BPS invariants for $\CC(Q,W)$, with part (4) replaced by 
\begin{itemize}
\item[(4')]  
There is an inequality $\dim(\CC(Q,W))\geq \sum_{\gamma}\dim(\BPS_{\CC(Q,W),\gamma})\lvert\gamma\lvert^2$.
\end{itemize}

\end{theorem}
\begin{proof}
The first statement follows directly from the second, which we now prove.  By Proposition \ref{nilpProp}, assuming that it is nonzero, $\phi_{\Tr(W)}\IC'_{\Msp^{\stable}_{\gamma}(Q)}\lvert_{\nilp}$ is equal to $\phi_{\Tr(W)}\IC_{U}$ for some analytic open $U\subset \Msp^{\stable}_{\gamma}(Q)$ containing $0_{\gamma}$ on which $\Tr(W)$ converges.  Then the result follows from the facts that $\IC_{U}$ is Verdier self-dual, and the vanishing cycles functor commutes with the Verdier duality functor, both of which are functors on $\MMHM(\Msp^{\stable}_{\gamma}(Q))$, and not merely the derived category; in other words, they are t exact for the perverse t structure.

The argument for part (3) is a categorified version of the argument in \cite{HuTo18} for Proposition \ref{HTprop} below.  Consider the left hand side of (\ref{MRiso2nilp}), with $n=2$.  The mixed Hodge module $(\phim{\Tr(W)}\IC_{\Msp^{2\sframed}_{\gamma}(Q)})_{\nilp}$ is supported on the moduli space of $A$-modules equipped with a surjection of $A$-modules from $A^{\oplus 2}$, and in particular, it has empty support after restricting to $\Msp^{2\sframed}_{\gamma}(Q)$ for any $\gamma$ with $\gamma_i> 2\cdot \underline{\dim}(A)_i$ for some $i\in Q_0$.  Since in each degree $\gamma$ we are considering a complex of monodromic mixed Hodge modules supported at a point, we can treat both sides of (\ref{MRiso2nilp}) as $\mathbb{N}^{Q_0}$-graded and cohomologicaly graded monodromic mixed Hodge structures.  

Passing to the underlying vector spaces via the symmetric monoidal functor $\rforg$, in degree $2\cdot \underline{\dim}(A)$, the left hand side of (\ref{MRiso2nilp}) is the vanishing cycle cohomology of a function with an isolated singularity, and in particular it is nonzero, while on the right hand side we have a free exterior algebra on a $\mathbb{N}^{Q_0}$-graded set of generators, with $2 \omega_{A,\gamma}\lvert \gamma\lvert$ generators in degree $\gamma$ (here the $2\lvert\gamma \lvert$ factor is the dimension of the cohomology of $\CC\mathbb{P}^{(2\lvert\gamma\lvert-1)}$).  For such an exterior algebra, the highest degree in which the algebra is nonzero is given precisely by the degree of the top exterior power, which is given by the right hand side in (4).  This also establishes the first part of (3) --- a graded exterior algebra is finite-dimensional if and only if it is finitely generated.

We next consider algebraic Jacobi algebras.  By our standing assumption, $\mathcal{S}$ is the Serre subcategory of $\CC(Q,W)$-modules containing those $M$ satisfying the condition that for all submodules $M'\subset M$ and all quotient modules $M\rightarrow M''$, the function $\Tr(W)$ is zero when evaluated on $M'$ and $M''$.  In particular, $\Tr(W)$ is zero when evaluated on $M$.  The argument for (1) and (2) is the same, now using Proposition \ref{simplesProp} to argue as above that the BPS cohomology for $\CC(Q,W)$ is concentrated in degree zero.  

We next consider part (4').  Write $\{T_1,\ldots, T_r\}$ for the simple modules, counted with multiplicity, occurring in the Jordan--Holder filtration of $A$, numbered so that only $T_1,\ldots T_{r'}$ are in $\mathcal{S}$.  In particular
\[
\underline{\dim}(A)=\sum_{n=1}^r\underline{\dim}(T_n).
\]
Then the left hand side of (\ref{MRiso2}) is supported on semisimple modules $\bigoplus_{p\in P} T_{n_p}$ admitting a surjection from $A^{\oplus 2}$, with each $n_p$ less than or equal to $r'$.  It follows that it has zero support if
\begin{equation}
\label{reducedIn}
\gamma_i> 2\sum_{n=1}^{r'}\underline{\dim}(T_n)_i
\end{equation}
for some $i\in Q_0$.  So the underlying vector space of the right hand side of (\ref{MRiso2}) is an exterior algebra, with $2\lvert\gamma\lvert\omega_{\CC(Q,W),\gamma}$ generators in degree $\gamma$, that is zero when restricted to any dimension vector $\gamma$ satisfying the inequality (\ref{reducedIn}).  Parts (3) and (4') follow as in the analytic case.
\end{proof}

Setting $z_1=z_2=q^{1/2}=1$ we deduce the following
\begin{corollary}
Let $A$ be a finite-dimensional analytic or algebraic Jacobi algebra associated to a symmetric quiver with potential.  Then all but finitely many of the BPS invariants $\omega_{A,\gamma}$ vanish, and they are all positive.
\end{corollary}
Recall that we call $W\in\CC Q_{\cyc}$ quasi-homogeneous if there is a grading of the arrows of $Q$ with nonnegative integers such that $W$ is homogeneous of degree $d>0$ with respect to the induced grading of $\CC Q_{\cyc}$.
\begin{corollary}
\label{pureCor}
Let $A=\CC(Q,W)$ be a finite-dimensional Jacobi algebra, for which $W$ is quasihomogeneous.  Then $\BPS_{\CC(Q,W),\gamma}$ is pure, and so the refined BPS invariants $\omega_{\gamma}(q^{1/2})$ belong to the subset $\mathbb{N}\subset\mathbb{Z}[q^{\pm 1/2}]$.
\end{corollary}
\begin{proof}
Fix a degree $\gamma\in\mathbb{N}^{Q_0}$, and assume that there is a simple $\gamma$-dimensional $\mathbb{C}Q$-module, otherwise the statements are trivial (in degree $\gamma$).  The (shifted) mixed Hodge module $\mathcal{F}=\IC_{\Msp_{\gamma}(Q)}$ is a pure mixed Hodge module on $\Msp_{\gamma}(Q)$.  The function $\Tr(W)$ is $\mathbb{C}^*$-equivariant, where $\Msp_{\gamma}(Q)$ is given the $\mathbb{C}^*$-action determined by the grading on the arrows $Q_1$.  Furthermore, $\phim{\Tr(W)}\mathcal{F}$ has proper support by Proposition \ref{simplesProp}.  In fact by $\CC^*$-invariance, we see that it is supported at the origin.  Then purity follows by \cite[Cor.3.2, Rem.3.4]{DMSS12}.  The statement regarding the $q$-refined BPS invariants follows since $\phim{\Tr(W)}\IC_{\Msp_{\gamma}(Q)}$ is a monodromic mixed Hodge module (i.e. lives in cohomological degree zero) supported at the point $0_{\gamma}\in\Msp_\gamma(Q)$, a zero-dimensional subscheme, and so $\HO(\phim{\Tr(W)}\IC_{\Msp_{\gamma}(Q)})$ is concentrated in degree zero.  By the definition of purity for complexes, it is pure of weight zero.
\end{proof}

Under the assumptions of Corollary \ref{pureCor}, the monodromy action on $\forg(\BPS_{\CC(Q,W),\gamma})$ is semisimple, and we can write
\[
\forg(\BPS_{\CC(Q,W),\gamma})\cong \mathcal{G}^0\oplus\mathcal{G}^{\neq 0}
\]
where the monodromy action on the first factor is trivial, and on the second factor has no invariant subrepresentation.  Then via Proposition \ref{mSprop}, Corollary \ref{pureCor} implies that $\mathcal{G}^0$ is pure of weight zero, while $\mathcal{G}^{\neq 0}$ is pure of weight -1.

Since in the quasihomogeneous case the Serre subcategory $\mathcal{S}$ contains all $\CC(Q,W)$-modules, we deduce from the proof of Theorem \ref{Jacthm} the following
\begin{corollary}
Let $A=\CC(Q,W)$ be a finite-dimensional Jacobi algebra, with $W$ quasihomogeneous.  Then there is an equality $\dim(A)=\sum_{\gamma} \dim(\BPS_{A,\gamma}) \lvert \gamma\lvert^2$
\end{corollary}

We next consider the implications of Theorem \ref{Jacthm} for the BPS Lie algebra.
\begin{corollary}
\label{ZB}
Let $A$ be either an algebraic or analytic Jacobi algebra associated to a symmetric quiver, and assume that $A$ is finite dimensional.  Then the BPS Lie bracket on $\mathfrak{g}_{A}^+$ vanishes.
\end{corollary}
\begin{proof}
The BPS Lie bracket defined in \cite{DaMe15b} preserves the cohomological grading, since it is defined to be the commutator bracket in the cohomological Hall algebra of \cite{COHA}.  The right hand side of (\ref{BPSalgdef}) lies entirely in cohomological degree one, since by Theorem \ref{Jacthm} each $\BPS_{A,\gamma}$ lies in cohomological degree zero, and the half Tate twist shifts degree by one.  Since the Lie bracket of two degree one elements lands in degree two, we conclude that the bracket vanishes.
\end{proof}

\begin{example}
\label{oneLoopEx}
Let $Q$ be the quiver with one vertex and one loop, which we will denote $X$, and let $W=X^{d+1}$.  Then $A=\CC(Q,W)$ is nilpotent, so we have also $A\cong \CC\{Q,W\}$.  The motivic Donaldson--Thomas theory of this quiver and potential was considered in \cite{DM11}.  

There is an isomorphism $\CC(Q,W)\cong\CC[X]/X^d$, which is the contraction algebra of a width $d$, $(0,-2)$ curve (indeed this fact was subsequently used in \cite{DM12} to calculate the motivic DT invariants of $(0,-2)$ curves).  In particular, this Jacobi algebra is finite, and Theorem \ref{Jacthm} applies.  Since a $(0,-2)$ curve has length one, only the first Gopakumar--Vafa invariant of such a curve is nonzero.  In particular, we deduce from Theorem \ref{Jacthm} that 
\begin{equation}
\label{OLm}
\BPS_{A,n}=0
\end{equation}
for $n\geq 2$.  On the other hand, the first cohomological BPS invariant is by definition equal to
\begin{equation}
\label{OLe}
\BPS_{A,1}:=\HO(\mathbb{A}^1,\phim{X^{d+1}}\QQQ_{\mathbb{A}^1})\otimes\LL^{-1/2}.
\end{equation}
Taking classes in the Grothendieck group of monodromic mixed Hodge structures, equations (\ref{OLm}) and (\ref{OLe}) together provide an alternative proof of the Hodge-theoretic version of \cite[Thm.1.1]{DM11}.  Assertion (4) of Theorem \ref{Jacthm} becomes the equality
\[
\dim(\CC[X]/X^d)=\dim(\HO(\mathbb{A}^1,\phim{X^{d+1}}\QQQ_{\mathbb{A}^1})).
\]
\end{example}
\begin{example}
The unmodified part (4) (as opposed to (4')) of Theorem \ref{Jacthm} can fail for algebraic Jacobi algebras.  For instance, take the same quiver as Example \ref{oneLoopEx}, but let $W\in X^3\ZZ[X]$ be a nonhomogeneous polynomial, which we write as $W=X^{e+1}D$ for some non-constant polynomial $D$ with nonzero constant term.  Then 
\[
\BPS_{A,1}\cong \HO(\mathbb{A}^1,\phim{X^{e+1}}\underline{\mathbb{Q}}_{\mathbb{A}^1})\otimes\LL^{-1/2}
\]
and so 
\[
\dim(\BPS_{A,1})=e<\dim\left(\CC(Q,W)\right).
\]
\end{example}

\begin{example}
Consider the doubled $\textrm{A}_2$ quiver
\[
\xymatrix{
1\ar@/_1pc/[rr]^x&& 2\ar@/_1pc/[ll]_y
}
\]
with potential $W=(xy)^{d+1}$ where $d\geq 1$.  Then the (algebraic) Jacobi algebra $A=\CC(Q,W)\cong \CC\{Q,W\}$ is finite-dimensional, and so Theorem \ref{Jacthm} applies.  By considering moduli stacks of semistable representations for nondegenerate stability conditions it is not too hard to calculate the cohomological BPS invariants explicitly, and find that
\begin{align*}
\BPS_{A,(0,1)}\cong \BPS_{A,(1,0)}&\cong \mathbb{Q}\\
\BPS_{A,(1,1)}&=\HO(\mathbb{A}^1,\phim{X^{d+1}}\QQQ_{\mathbb{A}^1})\otimes\LL^{-1/2}\\
\BPS_{A,\gamma}&=0\;\;\;\;\;\;\textrm{for}\;\gamma\notin\{(1,0),(0,1),(1,1)\}.
\end{align*}
Then Theorem \ref{Jacthm}(4) gives
\[
\dim(A)=1+1+d\cdot 2^2=4d+2,
\]
which one can verify by writing down the Jacobi algebra directly.
\end{example}

\section{Background on contractible curves}
\label{curvesec}

\subsection{Coherent sheaves on flopping curves}
As in the introduction, we fix a threefold flopping contraction $f\colon X\rightarrow Y$, where $p\in Y$ is an isolated singularity in the affine variety $Y$, and $C=f^{-1}(p)$ is a thickened rational curve.  

Inside the category $\Coh(X)$ of coherent sheaves on $X$ we will mostly be concerned with $\Coh_{\cpct}(X)$, the full subcategory containing those sheaves with compact support.  Since $Y$ is assumed to be affine, the set-theoretic support of any such $\mathcal{J}$ is a union of points of $X$, and the exceptional curve $C$.  It follows that any such $\mathcal{J}$ is obtained by taking iterated extensions of sheaves of the form $\OO_{C_{\reduced}}(d)$ for $d\in\mathbb{Z}$ and $\mathcal{O}_x$ for $x\in X$.  We define $\rk(\mathcal{J})$ to be the number of sheaves of the first kind appearing in some (equivalently, any) such filtration, in other words the length of $\mathcal{J}$ at the generic point of $C$.  We define the morphism

\begin{align}\label{chern}
\mathrm{cl}\colon \KK(\Coh_{\cpct}(X))\rightarrow &\mathbb{Z}^2\\\nonumber
[\mathcal{J}]\mapsto &(\rk(\mathcal{J}),\chi(\mathcal{J})).
\end{align}

We fix a divisor $D\subset X$ with $D\cdot C_{\reduced}>0$, i.e. pick a relatively ample divisor for $p$.  Then we define the slope of a compactly supported sheaf $\mathcal{J}$ by 
\[
\mu(\mathcal{J}):=\frac{\chi(\mathcal{J})}{ D\cdot \ch_{2}(\mathcal{J})}=(D\cdot C_{\reduced})^{-1} \frac{\chi(\mathcal{J})}{\rk(\mathcal{J})}
\]
following the usual convention that a sheaf with zero-dimensional support is defined to have slope $\infty$.  A sheaf $\mathcal{J}\in \Coh_{\cpct}(X)$ is called semistable if $\mu(\mathcal{J}')\leq \mu(\mathcal{J})$ for all proper subsheaves $\mathcal{J}'\subset \mathcal{J}$.  Since this definition does not change if we scale $\mu$ by a fixed constant, we could have defined $\mu(\mathcal{J})$ to be $\chi(\mathcal{J})/\rk(\mathcal{J})$.

Let $I$ be the ideal sheaf of $C_{\reduced}$ inside the hypersurface $X_0$ from Section \ref{flopping_curves_sec}, and let $I^{(i)}$ be its $i$th symbolic power, i.e. the ideal of functions vanishing to order $i$ along $C_{\reduced}$.  Then $C_i$ is defined to be the subscheme of $X$ defined by $I^{(i)}$.  We have $C_{\reduced}=C_1$ and $C=C_l$, where $l$ is the length of $C$.  It is proved in \cite[Lem.3.2]{Katz08} that for all $i\leq l$ there are equalities $\dim(\HO^0(X, C_i))=1$ and $\HO^1(X,C_i)=0$.  

Let $\Coh_C(X)$ denote the full subcategory of coherent sheaves with set-theoretic support on $C_{\red}$.  This category is equivalent to the category of coherent sheaves on the formal completion $\hat{X}$ of $X$ along $C_{\red}$.  Inside $\Coh(\hat{X})$ we consider the relative Serre twisting sheaf $\OO_{\hat{X}}(D')$ where $D'$ is a divisor satisfying $D'\cdot C_{\red}=1$, and the Serre twist $\mathcal{F}(1)\coloneqq \mathcal{F}\otimes\OO_{\hat{X}}(D')$.  We abuse notation by continuing to write $\mathcal{F}(1)$ for $\mathcal{F}\in\Coh(X)$ a coherent sheaf with set-theoretic support in $C_{\red}$.

\subsection{Derived equivalences}
\label{DEsec}
Let $\mathcal{C}\subset \Db(\Coh (X))$ be the full subcategory consisting of objects $\mathcal{J}$ such that $f_*\mathcal{J}=0$ (we remind the reader of our convention that all functors are derived).  Following Bridgeland \cite{Flops}, inside the category $\Coh(X)$ we consider the torsion structure given by
\begin{align*}
\Torsion_{-1}=&\{\mathcal{J}\in \Coh(X) \textrm{ such that }\Ho^1(f_*\mathcal{J})=0\textrm{ and }\Hom(\mathcal{J},\mathcal{C})=0\}\\
\Tfree_{-1}=&\{\mathcal{J}\in\Coh(X)\textrm{ such that }\Ho^0(f_*\mathcal{J})=0\}.
\end{align*}
The Abelian subcategory ${}^{-1}\!\Per(X/Y)$ is defined to be the full subcategory of objects $\mathcal{J}\in \mathrm{Ob}(\Db(\Coh(X)))$ such that
\[
\Ho^{-1}(\mathcal{J})\in \Tfree_{-1},\quad\quad
\Ho^0(\mathcal{J})\in\Torsion_{-1},\quad\quad\Ho^{i\neq 0,1}(\mathcal{J})=0.
\]

Let $\mathcal{L}$ be ample relative to $p$, and let
\begin{equation}
\label{Ndef}
0\rightarrow \OO_X^{\oplus r}\rightarrow\mathcal{M}\rightarrow \mathcal{L}\rightarrow 0
\end{equation}
be the extension associated to a minimal generating set of $\HO^1(X,\mathcal{L}^{-1})$ as a $\HO^0(Y,\OO_Y)$-module, and set $\mathcal{P}=\OO_X\oplus\mathcal{M}$.
\begin{proposition}\cite[Thm.A]{FlopsAndNoncommRings}
\label{VdBequiv}
The vector bundle $\mathcal{P}$ is a projective generator for ${}^{-1}\!\Per(X/Y)$, and there is an equivalence of categories
\[
\Psi=\Hom_{Y}(\mathcal{P},-)\colon \Db(\Coh(X))\rightarrow \Db(A\rmod)
\]
where $A:=\End_X(\mathcal{P})$, inducing an equivalence of categories between ${}^{-1}\!\Per(X/Y)$ and $A\rmod$.
\end{proposition}

We denote by $\Phi=-\otimes{}_A\mathcal{P}$ the quasi-inverse to $\Psi$.  Write ${}^{-1}\!\Per_C(X/Y)$ for the category of perverse coherent sheaves set-theoretically supported on $C$.  As in \cite{FlopsAndNoncommRings} it is often useful to consider instead the category ${}^{-1}\!\Per(\hat{X}/\hat{Y})$ where $\hat{Y}$ is obtained by completing at the ideal of functions vanishing at $0$, and $\hat{X}$ is obtained by completing at the ideal sheaf $\mathcal{I}_{C}$.  The completion of modules then induces a functor ${}^{-1}\!\Per(X/Y)\rightarrow {}^{-1}\!\Per(\hat{X}/\hat{Y})$, which is fully faithful and preserves subobjects, when restricted to ${}^{-1}\!\Per_C(X/Y)$.  Then the following proposition follows from Proposition 3.5.7 in \cite{FlopsAndNoncommRings}.

\begin{proposition}\cite{FlopsAndNoncommRings}
The simple objects of ${}^{-1}\!\Per_C(X/Y)$ are 
\begin{align*}
S_1:=&\OO_C\\
S_2:=&\OO_{C_{\reduced}}(-1)[1].
\end{align*}
\end{proposition}
\begin{proposition}
\label{no_exotics}
Let $\mathcal{J}$ be a compactly supported coherent sheaf on $Y$ satisfying $\chi(\mathcal{J})=0$.  Then $\mathcal{J}$ is semistable if and only if it admits a filtration
\[
0=\mathcal{J}_0\subset\mathcal{J}_1\subset\ldots\subset \mathcal{J}_n=\mathcal{J}
\]
with each subquotient $\mathcal{J}_i/\mathcal{J}_{i-1}$ isomorphic to $\mathcal{O}_{C_{\reduced}}(-1)$.
\end{proposition}
\begin{proof}
The if part follows from the closure of the category of semistable objects of fixed slope under taking extensions, so we focus on the only if part.  A compactly supported sheaf $\mathcal{J}$ decomposes as a direct sum $\mathcal{J}=\mathcal{J}_0\oplus \mathcal{J}'$ where $\mathcal{J}_0$ is set theoretically supported on $C$, and $\mathcal{J}'$ is set-theoretically supported away from $C$.  Then $\chi(\mathcal{J}')$ is the length of $\mathcal{J}'$, and is in particular nonzero if $\mathcal{J}'$ is.  So by semistability, $\mathcal{J}$ is set-theoretically supported on $C$.

We may write $\Psi(\mathcal{J})\cong [\ldots\rightarrow 0\rightarrow M_0\xrightarrow{d} M_1\rightarrow 0\rightarrow \ldots]$, where $\ker(d)\in \Psi(F_{-1})$ and $\coker(d)\in \Psi(T_{-1})$ are nilpotent modules.  If $d$ is not injective, there must be an inclusion $S_i\hookrightarrow \ker(d)$ for one of $i=1,2$, giving rise either to a nonzero map $\OO_{C_{\reduced}}(-1)[1] \rightarrow \mathcal{J}$ or a nonzero map $\OO_{C}\rightarrow \mathcal{J}$.  The first is not possible since $\mathcal{J}$ is a coherent sheaf, while the second is ruled out by semistability and $\chi(\OO_C)=1$.  So we deduce that $\Psi(\mathcal{J})\cong M[-1]$ for some nilpotent $A$-module.  Since $M$ is nilpotent it is obtained by iterated extensions from the two modules $S_1$ and $S_2$.  Since $\chi(\Phi(M))=0$, it follows that $M$ is obtained by taking iterated extensions of the simple $S_2=\Psi(\OO_{C_{\reduced}}(-1)[1])$, and in particular admits a filtration
\[
0=M_0\subset M_1\subset\ldots\subset M_n=M
\]
where each subquotient is isomorphic to $S_2$.  Applying $\Phi[-1]$ to this filtration, we obtain the desired filtration for $\mathcal{J}$.
\end{proof}

\subsection{Moduli spaces of semistable coherent sheaves}
The construction and basic properties of moduli spaces of semistable shaves on algebraic varieties is treated generally in \cite{HL}.  In the case of sheaves on contractible curves, these constructions are somewhat simplified.
\begin{lemma}
Let $\mathcal{J}_1$ and $\mathcal{J}_2$ be coherent sheaves on an algebraic variety, generated by global sections, with $\HO^1(X,\mathcal{J}_i)=0$ for $i=1,2$, and let 
\[
0\rightarrow \mathcal{J}_1\rightarrow \mathcal{J}\rightarrow \mathcal{J}_2\rightarrow 0
\]
be a short exact sequence.  Then $\mathcal{J}$ is generated by global sections and satisfies $\HO^1(X,\mathcal{J})=0$.
\end{lemma}
\begin{proof}
This is a diagram chase, using the morphism of distinguished triangles
\[
\xymatrix{
\HO(X,\mathcal{J}_1)\otimes \OO_X\ar[r]\ar[d]&\HO(X,\mathcal{J})\otimes\OO_X\ar[r]\ar[d]&\HO(X,\mathcal{J}_2)\otimes\OO_X\ar[d]\ar[r]&\\
\mathcal{J}_1\ar[r]&\mathcal{J}\ar[r]&\mathcal{J}_2\ar[r]&.
}
\]
\end{proof}
Combining with Proposition \ref{no_exotics}, we deduce the following
\begin{corollary}
Let $X\rightarrow Y$ be a flopping curve contraction.  Let $\mathcal{J}$ be a semistable coherent sheaf on $X$ with $\chi(\mathcal{J})=0$.  Then $\mathcal{J}$ is set-theoretically supported on $C_{\red}$, the sheaf $\mathcal{J}(1)$ is generated by global sections, and $\dim(\Hom(\OO_{X},\mathcal{J}(1)))=\rk(\mathcal{J})$.
\end{corollary}
Now fix $r$, and a $r$-dimensional vector space $\mathbb{V}$.  The Grothendieck quot scheme $\Quot(\mathbb{V},(r,r))$ is defined to be the fine moduli scheme parameterising quotients
\[
h'\colon \OO_X\otimes_{\CC} \mathbb{V}\twoheadrightarrow \mathcal{J}'
\]
where $\chi(\mathcal{J}')=r=\rk(\mathcal{J}')$.  We consider the open subscheme $\Quot^{\circ}(\mathbb{V},(r,r))\subset \Quot(\mathbb{V},(r,r))$ given by the condition that the target sheaf is semistable, and $\HO^0(h')$ is surjective.  Note that semistability forces the set-theoretic support to lie in $C_{\red}$, for otherwise $\mathcal{J}'$ splits as a direct sum, with one summand of infinite slope.  Applying the relative Serre twist, $\Quot^{\circ}(\mathbb{V},(r,0))$ is isomorphic to the stack of pairs $(\mathcal{J},h)$, where $\mathcal{J}=\mathcal{J}'(1)$ is a semistable sheaf with $\rk(\mathcal{J})=r$ and $\chi(\mathcal{J})=0$, and $h'\lvert_{\mathbb{V}}$ is a choice of basis for $\HO^0(X,\mathcal{J}(1))$, where we embed $\mathbb{V}\subset \mathcal{O}_X\otimes_{\mathbb{C}} \mathbb{V}$ via the embedding $\mathbb{C}\subset \mathcal{O}_X$ given by the constant functions.  

Letting $\Gl_r(\CC)$ act on $h'$ via precomposition, we obtain the isomorphism
\[
\Mst_{r,0}^{\sst}(X)\cong \Quot^{\circ}(\mathbb{V},(r,0))/\Gl_r(\CC),
\]
i.e. the moduli stack of sheaves is a global quotient stack.  The points of the coarse moduli space $\Msp_{r,0}^{\sst}(X)$ are given by S-equivalence classes of sheaves, or equivalently, polystable sheaves.  In particular, by Proposition \ref{no_exotics} we deduce that $\left(\Msp_{r,0}^{\sst}(X)\right)_{\reduced}\cong \pt$.  Likewise, the Chow variety $\Chow_X(r[C_{\reduced}])$ is isomorphic to a point, and the natural map 
\begin{equation}
\label{chowiso}
\pt=\left(\Msp_{r,0}^{\sst}(X)\right)_{\reduced}\rightarrow \Chow_X(r[C_{\reduced}])=\pt
\end{equation}
is obviously an isomorphism.

If we fix the Euler characteristic to be one, the geometry becomes even more well-behaved.  The following theorem is due to Sheldon Katz.
\begin{theorem}\cite[Proof of Prop 3.3]{Katz08}
\label{Katzprop}
Let $\Msp_{r,1}^{\sst}(X)$ be the \textit{fine} moduli space of stable one-dimensional sheaves $\mathcal{J}$ with $\chi(\mathcal{J})=1$ and $\rk(\mathcal{J})=r$.  Then for $r\leq l$, $\Msp_{r,1}^{\sst}(X)$ has exactly one closed point, corresponding to the sheaf $\OO_{C_r}$, otherwise it is empty.  The length of the structure sheaf $\mathcal{O}_{\Msp_{r,1}^{\sst}(X)}$ is $n_{C,r}$.
\end{theorem}

\subsection{The contraction algebra}
\label{CAsec}
Before reminding the reader of the definition of the contraction algebra from \cite{DW16}, we recall the type of deformation functor that it represents.  Let $\Art_1$ be the category of pointed finite-dimensional algebras, i.e. finite-dimensional $\mathbb{C}$-algebras $\Gamma$ equipped with a retraction of algebras $p\colon\Gamma\rightarrow\CC$, such that $\ker(p)^N=0$ for $N\gg 0$.  Let $\mathcal{A}$ either be $\QCoh(U)$ for $U$ a quasi-projective scheme $\CC$-scheme, or $\Mod(\Lambda)$ for some algebra $\Lambda$.  Let $a\in\mathrm{ob}(\mathcal{A})$.  Then the noncommutative deformation functor $\Def_a^{\mathcal{A}}$ is defined to be the functor taking an element $(\Gamma,p)$ of $\Art_1$ to the set of isomorphism classes of triples $(b,\tau, \delta)$ where $b\in\mathcal{A}$, $\tau:\Gamma\rightarrow \Hom_{\mathcal{A}}(b)$ is a homomorphism of algebras, and $\delta\colon(\Gamma/\ker(p))\otimes_{\Gamma}b\rightarrow a$ is an isomorphism, such that $-\otimes_{\Gamma} b\colon \Gamma\rmod \rightarrow \mathcal{A}$ is exact.  See \cite{DW16} for details regarding the notion of isomorphism of deformations, and \cite{La02,Er07,Seg08,KSDef} for more treatment of noncommutative deformation theory.

Write $\hat{R}$ for the ring of functions $\Gamma(Y)$ completed at the ideal $\mathcal{I}_0$ corresponding to the singular point $0$.  Let $\mathcal{L}'$ be a line bundle on $\hat{X}$ such that $\mathcal{L}'\cdot C_{\reduced}=1$, and define $\mathcal{M}'$ with respect to the short exact sequence 
\[
0\rightarrow \OO_{\hat{X}}^{\oplus r'}\rightarrow\mathcal{M}'\rightarrow\mathcal{L}' \rightarrow 0
\]
analogous to (\ref{Ndef}), and define $\mathcal{P}'=\OO_{\hat{X}}\oplus\mathcal{M}'$ and $\hat{A}:=\End_{\hat{X}}(\mathcal{P}')=\End_{\hat{Y}}(\hat{R}\oplus M)$ where $M=p_*\mathcal{M}'$.  Van den Bergh shows in \cite[Sec.3.4]{FlopsAndNoncommRings} that by picking such an $\mathcal{L}'$, and feeding it back into \cite[Sec.3.2]{FlopsAndNoncommRings} we get an equivalence of categories between ${}^{-1}\!\Per(\hat{X}/\hat{Y})$ and $\hat{A}\rmod$ as in Proposition \ref{VdBequiv}.
\begin{definition}\cite{DW16}
The \textit{contraction algebra} $A_{\con}$ is the quotient $\hat{A}/I$, where $I$ is the two sided ideal of endomorphisms of $\mathcal{P}'$ factoring through the summand $\hat{R}$.
\end{definition}
Defining $S=\Hom_{\hat{X}}(\mathcal{M}',\OO_{C_{\reduced}})$ we obtain a simple $A_{\con}$-module, which is $S_2$ when considered as a $\hat{A}$-module via the surjection $\hat{A}\rightarrow A_{\con}$.

\begin{theorem}\cite{DW16}
\label{DWthm}
The contraction algebra $A_{\con}$ is finite-dimensional, and represents the isomorphic deformation functors $\Def^{\QCoh(Y)}_{\OO_{C_{\reduced}}}$, $\Def^{A_{\con}\rmod}_{S}$ and $\Def^{\hat{A}\rmod}_{S_2}$.
\end{theorem}

\begin{proposition}
\label{ConAl}
The contraction algebra $A_{\con}$ is an analytic Jacobi algebra for a quiver $Q$ with analytic potential $W\in K\{Q\}_{\cyc}$.  The quiver $Q$ has one vertex, and zero, one, or two loops, depending on whether $C$ is of $N$-type $(-1,-1)$, $(0,-2)$ or $(1,-3)$.
\end{proposition}
\begin{proof}
By \cite{DW16}, $A_{\con}$ represents the deformation functor $\Def^{\hat{A}\rmod}_{S_2}$.  Let $(\Ext(S_2,S_2),b_{\bullet})$ be the cyclic $A_{\infty}$-endomorphism algebra of $S_2$, and let $W(x)=\sum_{i\geq 2}\frac{1}{i+1}\langle b_{i}(x,\ldots,x),x\rangle$ be the formal function on $V=\Ext^1(S_2,S_2)$ arising from the cyclic $A_{\infty}$-structure.  Then it is well-known (see e.g. \cite{Seg08, KSDef, Koszul}) that $\Def^{\hat{A}\rmod}_{S_2}$ is represented by the zeroeth cohomology of the Koszul dual of $(\Ext(S_2,S_2),b_{\bullet})$, which is by construction the Ginzburg dga $\hat{\Gamma}(Q,W)$ for the pair $Q,W$ (see \cite{ginz}), where $Q$ is a quiver with arrows given by a basis for $V$.  The analyticity of the potential $W$ is given by \cite[Lem.4.1]{To18}.
\end{proof}

\begin{remark}
Since $A_{\con}$ is a finite-dimensional formal Jacobi algebra, in principle we could have used \cite[Thm.3.16]{HuZh18} to produce a formal isomorphism $G$ of the completed free path algebra $\widehat{\mathbb{C} Q}$ such that $G_*W$ is algebraic, and considered the cohomological DT theory of $\mathbb{C}(Q,G_*W)$, which is given in terms of algebraic (not analytic) mixed Hodge modules.  The automorphism $G$ induces an automorphism $G_{\gamma}$ of the formal completion of $\Rep_{\gamma}(Q)$ around the nilpotent locus.  The resulting motivic Donaldson--Thomas invariant, defined for the possibly formal function $\Tr(W)$ via \cite[Rem.3.4]{BuJoMe13}, is the same for $G_*W$ and $W$, since the motivic vanishing cycles are determined by spaces of arcs from the nilpotent locus, and $G_{\gamma}$ acts on these via isomorphisms.  We prefer to work at the level of analytic potentials, since the question of whether $G$ can be chosen to be convergent in a sense that enables us to extend $G_{\gamma}$ to an analytic neighbourhood of the nilpotent locus seems to be open\footnote{In fact this question has subsequently been settled in the affirmative by Hua and Keller; see \cite[Thm.4.4]{HuKe19}.}.  In any case we will want to be able to consider vanishing cycles for infinite-dimensional Jacobi algebras when we come to Conjecture \ref{SRC}, where the results of \cite{HuZh18} no longer apply.
\end{remark}
\begin{remark}
The question of \textit{which} finite-dimensional Jacobi algebras arise as contraction algebras is open, but according to a conjecture of Brown and Wemyss, for any potential $W$ on a select list of symmetric quivers giving rise to finite $A=\CC\{Q,W\}$, the algebra $A$ is a contraction algebra for some curve.  See \cite{BW21} for background and further developments regarding this conjecture.  In this paper we deal only with irreducible contractible curves, so in fact the list of quivers we are concerned with is quite short: the zero loop quiver, the one loop quiver, and the two loop quiver.  In the first two cases the conjecture is clearly true, but for the two loop quiver it is still open.  We refer to \cite{HuKe18,Iy20,IySh18,IySm18,BW21} for some recent classification results on finite-dimensional Jacobi algebras.
\end{remark}

\subsection{From the contraction algebra to Gopakumar--Vafa invariants}

We denote by $\Msp_r^{m\sframed} (A_{\con})$ the fine moduli space of pairs $(M,(v_1,\ldots,v_m))$ where $M$ is a $r$-dimensional $A_{\con}$-module and $v_1,\ldots,v_m\in M$ generate $M$ under the action of $A_{\con}$.  If we set $m=1$ then this is the closed subscheme of $\ncHilb_{\gamma}(Q)$, the points of which correspond to stable framed $A$-modules.  Due to the parity condition in Theorem \ref{MRthm} and Proposition \ref{MRc1}, we will instead focus on the case in which $m\geq 2$ is even.  The scheme $\Msp_r^{m\sframed} (A_{\con})$ is  typically highly singular; in Donaldson--Thomas theory, when confronted with a singular complex scheme $T$, the basic enumerative invariant is not quite the Euler characteristic of the underlying scheme, but the Euler characteristic weighted by the Behrend function $\nu_T$.  The following theorem relates this measure to cohomological DT theory:

\begin{theorem}\cite{PaPr01}.
\label{PPProp}
Let $f\in \Gamma(T')$ be a holomorphic function on a smooth complex manifold, and assume that the complex variety $T$ is isomorphic to $\crit(f)$ as an analytic space.  Then
\[
\chi(T,\nu_T)=\chi(\HO(T',\phi_f\IC_{T'})).
\]
\end{theorem}
This is stated in Behrend's original paper \cite[Sec.1.2]{Behr09}, which dealt with algebraic stacks.  The original proof of \cite[Cor.2.4]{PaPr01} is for analytic subspaces.  It follows that if, with the assumptions of Theorem \ref{PPProp} the underlying topological space of $T$ is a single isolated point $p$, then $\nu_T(p)$ is equal to the Milnor number of $f$, i.e. 
\[
\nu_T(p)=\dim_{\CC}(\OO_T).  
\]

Finally, we state the relation between weighted Euler characteristics and the Gopakumar--Vafa invariants.  This is a slight generalisation of \cite[Thm.4.4]{HuTo18}.  

\begin{theorem}\cite{HuTo18}
\label{HTprop}
Write $\mathcal{Z}^m(t):=\sum_{i\geq 0}\chi(\Msp_i^{m\sframed} (A_{\con}),\nu_{\Msp_i^{m\sframed} (A_{\con})})t^i$.  Then there is an equality of generating series
\[
\mathcal{Z}^m(t)=\prod_{1\leq i\leq l}(1-(-1)^{m\cdot i} t^i)^{m\cdot i\cdot n_{i}}.
\]
\end{theorem}
The proof is a minor adaptation of the proof of \cite[Thm.4.4]{HuTo18}, and so we only sketch it: see \cite{HuTo18} for more details.
\begin{proof}
Recall that we define $\mathcal{C}$ to be the full subcategory of objects such that $f_*\mathcal{J}=0$.  Let $\mathcal{C}_0=\mathcal{C}\cap \Coh(X)$.  Let $\mathcal{L}'_1,\ldots,\mathcal{L}'_n$ be line bundles defined by divisors $D_1,\ldots,D_n$ on $\hat{X}$, each intersecting $C$ at pairwise disjoint points, and satisfying $D_i\cdot C=1$.  Then Van den Bergh's equivalence $\Psi$ of Proposition \ref{VdBequiv} induces the equivalence $\Psi\colon \mathcal{C}_0\rightarrow A_{\con}\rmod$.  For each $i=1,\ldots,n$ there is a diagram of functors, commuting up to natural equivalence
\[
\xymatrix{
\ar[dr]_{\otimes \mathcal{O}_{D_i}}\mathcal{C}_0\ar[rr]^-{\Psi}&&A_{\con}\rmod\ar[dl]^{\mathrm{forg}}\\
&\mathrm{Vect}.
}
\]
So the data of an $A_{\con}$-module $\rho$ and $n$ elements of the underlying vector space is equivalent to the data of a coherent sheaf $\mathcal{J}$ in $\mathcal{C}_0$ along with a section $s$ of $\mathcal{J}\otimes \mathcal{O}_{\hat{X}}(\sum_{i\leq n} D_i)$.  As in \cite[Thm.4.4]{HuTo18}, the condition on this data to define a surjection $A_{\con}^{\oplus n}\rightarrow \rho$ is the same as the condition for the pair $(\mathcal{J},s)$ to be a parabolic stable pair, in the terminology of \cite{To15}.  Then the result is a direct application of \cite[Prop.3.16]{To15}.
\end{proof}

\section{Refined invariants for contractible curves}
\label{finalsec}
\subsection{Refined Gopakumar--Vafa invariants for contractible curves, \`a la Katz \cite{Katz08}}
\label{alaK}
We start by giving an approach to refined Gopakumar--Vafa invariants arising from the combination of the original definition of Katz for the Gopakumar--Vafa invariants of a flopping curve, with the work of Joyce et al.  

Recall that by part of Katz's work, the reduced subscheme of the fine moduli scheme $\Msp^{\stable}_{r,1}(X)$ is a single point (Theorem \ref{Katzprop}).  The unreduced scheme is a fine moduli scheme, and the map $\Msp^{\stable}_{r,1}(X)\rightarrow \Mst^{\stable}_{r,1}(X)$ to the stack of stable sheaves is a $\mathbb{C}^*$-bundle.  By \cite[Cor.2.13]{PTVV} this latter stack is the underlying algebraic stack of a (-1)-shifted symplectic stack.  In \cite{BBJ18} the authors use this fact to endow $\Msp^{\stable}_{r,1}(X)$ with a $d$-critical locus description in the sense of \cite{Joy13}.  Since the underlying scheme is a point, and so any neighborhood of the point is the entire scheme, it follows also from \cite[Thm.5.18]{BBJ18} that $\Msp^{\stable}_{r,1}(X)$ is the scheme-theoretic critical locus of a function $g$ on a smooth algebraic variety $U$.  The virtual canonical bundle $K=K_{M,s}$ on $M=(\Msp^{\stable}_{r,1})_{\reduced}$ is then a line bundle on $M$, which is clearly trivial (since $M$ is a point), so that we can pick an isomorphism $\mathcal{O}_{M_{\reduced}}^{\otimes 2}\cong K_{M,s}$.  In the language of \cite{Joy13}, this d-critical scheme admits trivial orientation.  We abbreviate $L=\mathcal{O}_{M_{\reduced}}$.  Given the data of the oriented $d$-critical scheme $(M,s,L)$, \cite[Thm.6.9]{Br12} constructs a canonical Verdier self-dual monodromic mixed Hodge module $\Phi_{M,s,L}$ on $M$.  By construction, this is just the monodromic mixed Hodge module $\phim{g}\IC_U$.  Considered as a cohomologically graded monodromic mixed Hodge structure, it is concentrated in degree zero.

The above discussion establishes part (1) of the following theorem, while part (2) follows by Theorem \ref{Katzprop} and Theorem \ref{PPProp}.  Part (3) is a consequence of the definition of the dimension of a monodromic mixed Hodge structure, and part (1).
\begin{theorem}
\label{KGVprop}
Let
\[
\GV'_{C,r,0}:=\HO(\Msp^{\stable}_{r,1},\Phi_{M,s,L})
\]
be the genus zero cohomological Gopakumar--Vafa invariants\footnote{The subscript zero for the Gopakumar--Vafa invariant refers to the genus, the subscript 1 on the right hand side refers to the Euler characteristic of the coherent sheaves we are considering.} defined for the flopping curve $C$ as above.  Then
\begin{enumerate}
\item
The invariants $\GV'_{C,r,0}$ are Verdier self-dual complexes of monodromic mixed Hodge structures, concentrated in cohomological degree zero (i.e. they are monodromic mixed Hodge structures).
\item
There is an equality $\dim(\GV'_{C,r,0})=n_{C,r}$.
\item
If $n_{C,r}$ vanishes, then so does $\GV'_{C,r,0}$.
\end{enumerate}
\end{theorem}

\subsection{Refined Gopakumar--Vafa invariants from BPS cohomology; proof of Theorem \ref{ThmA}}
\label{ThmA_proof_sec}
In \cite{MaTo18} a proposal is given for the definition of (all genus) Gopakumar--Vafa invariants, which we briefly recall.  Let $\beta\in\HO_2(Y,\mathbb{Z})$ be a homology class.  Then $M=\Msp_{(\beta,1)}(Y)$ is a fine moduli space, and moreover a d-critical scheme, which is assumed to carry a special kind of orientation $K^{1/2}$ called Calabi--Yau orientation data (see the appendix to \cite{MaTo18}).  Then as above, $M=\Msp^{\sstable}_{(\beta,1)}(Y)$ is a fine moduli scheme, and we obtain the sheaf $\Phi_{M,s,K^{1/2}}\in \Perv(M)$.  Maulik and Toda define the numbers $n_{g,\beta}$ via
\[
\sum_{i\in \mathbb{Z}}\chi (\Ho^i(\pi_*\Phi_{M,s,K^{1/2}}))y^i=\sum_{g\geq 0} n_{g,1,\beta}(y^{1/2}+y^{-1/2})^{2g}.
\]
Here $\pi$ is the Hilbert-Chow map.  This map is projective, and the sheaf $\Phi_{M,s,K^{1/2}}$ is Verdier self-dual, so that its direct image along $\pi$ is also Verdier self dual, and the definition makes sense (i.e. the left hand side is invariant under $y\mapsto y^{-1}$).

The correct way to generalise this definition for non-primitive classes in $\mathbb{Z}^2$, i.e. classes of the form $(r,0)$, as opposed to $(r,1)$, is to replace the sheaf $\Phi_{M,s,K^{1/2}}$ by the BPS sheaf (see \cite{To17} for an application of this definition to wall crossing), i.e. we first define
\[
\BPSs_{\Mst^{\sstable}_{r,\beta}(Y)}:=\Ho^1(p_{r,\beta,*}\Phi_{M,s,K^{1/2}})
\]
where
\[
p_{r,\beta}\colon \Mst_{r,\beta}^{\sstable}\rightarrow \Msp_{r,\beta}^{\sstable}
\]
is the map from the moduli stack to the coarse moduli space, and then define
\[
\sum_{i\in \mathbb{Z}}\chi (\Ho^i(\pi_*\BPSs_{\Mst^{\sstable}_{r,\beta}(Y)}))y^i=\sum_{g\geq 0} n_{g,r,\beta}(y^{1/2}+y^{-1/2})^{2g}.
\]

By Proposition \ref{ConAl}, the contraction algebra is an analytic Jacobi algebra: $A_{\con}\cong \CC\{Q,W\}$ for some analytic potential $W$.  By Theorem \ref{DWthm} $A_{\con}$ is finite-dimensional, and so by Proposition \ref{nilpProp}, the stack of $\CC\{Q,W\}$-modules is an open and closed substack of any open analytic neighbourhood in which the function $\Tr(W)$ is defined, and we can define the cohomological BPS invariants for $A_{\con}$ as in Definition \ref{BPSdef}.  By \cite{To18} there is an isomorphism of stacks
\begin{equation}
\label{derIso}
\Rep_r(\CC\{Q,W\})\cong \Mst_{r,0}^{\sstable}(Y).
\end{equation}
The stack $\Rep_r(\CC\{Q,W\})$ has a d-critical structure\footnote{The question of whether the isomorphism (\ref{derIso}) preserves the oriented d-critical structures on either side is still open, and will be returned to in future work.  See, for example, the discussion around \cite[Rem.A.1]{To17}.} $s'$ coming from its presentation as a global critical locus, with a canonical orientation, which is moreover Calabi--Yau in the sense of Maulik and Toda (see \cite[Thm.6.4.2]{Dav10a}), and we therefore obtain an oriented d-critical structure on $\Mst_{r,0}^{\sstable}(Y)$, for which the resulting monodromic mixed Hodge module $\Phi_{M,K,s'}$ is the vanishing cycle complex of the function $\Tr(W)$.  By \cite{DaMe15b} there is an isomorphism
\[
\Ho^1\left((\Rep_r(\CC Q)\xrightarrow{p} \Msp_r(Q))_*\underline{\mathbb{Q}}_{\Rep_r(\CC Q)}\right)\cong \IC_{\Msp^{\stable}_r(Q)}
\]
and so 
\[
\HO(\Msp_r(Q),\Ho^1(p_*\Phi_{M,K,s})')\cong \BPS_{\mathbb{C}\{Q,W\},r}.
\]
Putting this all together, we see that the BPS cohomology for the stack of semistable sheaves on $X$ of class $(r,0)$, with the d-critical structure coming from the isomorphism (\ref{derIso}), is given by
\[
\BPS_{\Mst^{\sstable}_{(r,0)}(X)}=\BPS_{A_{\con},r}.
\]

\begin{proposition}
\label{FL}
There is an equality for all $r\in\mathbb{N}$
\[
\dim(\BPS_{A_{\con},r})=n_{C,r}
\]
\end{proposition}
\begin{proof}
Setting $n=2$ in (\ref{MRiso2}) the graded Euler characteristic of the left hand side is given by
\begin{align*}
\sum_{i\geq 0}\chi\left(\HO(\Msp^{2\sframed}_i(Q),\phim{\Tr(W)}\IC_{\Msp^{2\sframed}_i(Q)}\lvert_{\nilp})\right)t^i=&\sum_{i\geq 0}\chi \left(\Msp^{2\sframed}_i(A_{\con}),\nu_{\Msp^{2\sframed,\nilp}_i(A_{\con})}\right)t^i\\ =&\mathcal{Z}^2(t)
\end{align*}
by Theorem \ref{PPProp}.  The graded Euler characteristic of the right hand side of (\ref{MRiso2}) is given by
\[
\prod_{1\leq i\leq l}(1- t^i)^{2i\cdot \dim(\BPS_{A_{\con},i})}
\]
and so the result follows directly from Theorem \ref{HTprop}.
\end{proof}

We have seen that $\Msp_{r,0}^{\sstable}(Y)$ is isomorphic to a point, as is the Chow variety.  Also, $\pi_*\BPSs_{\Mst^{\sstable}_{r,0}(Y)}$ is a perverse sheaf, and so all higher genus Gopakumar--Vafa invariants vanish, and we finally define
\begin{definition}
\label{GV_co_def}
We define the genus zero Gopakumar--Vafa cohomology for the contractible curve $C$ by setting
\[
\GV_{C,r,0}=\BPS_{A_{\con},r}.
\]
\end{definition}
Then Theorem \ref{ThmA} follows from Proposition \ref{FL} and Theorem \ref{ThmB}.
\begin{remark}
Since Theorem \ref{ThmB} identifies the dimension of $\GV_{C,r,0}$ with $n_{C,r}$, the analogue of part (3) of Theorem \ref{KGVprop} holds.  We emphasize this since, a priori, it is very difficult to demonstrate vanishing of refined invariants from vanishing of numerical ones.  Were this not the case, the paper \cite{DM12} calculating motivic DT invariants of $(0,-2)$ curves would have been very short.  This vanishing result, along with the vanishing
\[
n_{C,r}=0\quad \textrm{for }r>l(C)
\]
means that it is possible to calculate refined invariants by hand.  
\end{remark}
\begin{remark}
For an application of this vanishing, see the paper \cite{VG19}, which extends \cite{BrWe17} to the refined setting.  In the earlier paper, a pair of length 2 flopping curves are given that have the same Gopakumar--Vafa invariants, but which are non-isomorphic, and are distinguished by their contraction algebras.  In \cite{VG19} it is shown that these flopping curves are not distinguished by their refined Gopakumar--Vafa/BPS invariants either.
\end{remark}
\begin{definition}
We define successively more refined invariants $n_{C,r}(q^{1/2})$, $n_{C,r}(z_1,z_2)$, $n_{C,r,\MMHS}$ by taking $\wtm$, $\wth$ and the class in the Grothendieck group $\KK(\MMHS)$, respectively, of $\GV_{C,r,0}$.  These specializations are as defined in Section \ref{specSec}.
\end{definition}
The following is a special case of Corollary \ref{pureCor}.
\begin{proposition}
Let $A_{\con}\cong \CC(Q,W)$, where $W$ is quasi-homogeneous.  Then $n_{C,r}(q^{1/2})\in\mathbb{N}$.
\end{proposition}
The above proposition is enough to make one wonder whether $q$-refined Gopakumar--Vafa invariants carry any extra information at all:
\begin{conjecture}
\label{purityQ}
The monodromic mixed Hodge structure $\GV_{C,r,0}$ is pure for every flopping curve $C$, and every $r$.  Equivalently $n_{C,r}(q^{1/2})=n_{C,r}$ for all $C$ and $r$.
\end{conjecture}
This conjecture is open even for $r=1$.  For instance it is known (e.g. see \cite{PoMO}, or \cite[Sec.4]{AC73}) that there exist functions $f(x,y)$ in two variables with an isolated singularity at the origin, such that the monodromy action on the cohomology of the Milnor fibre is not semisimple (this is equivalent to impurity of vanishing cycle cohomology, since by definition \cite[(5.1.6.2)]{Sa88} the weight filtration on the vanishing cycle cohomology is the monodromy filtration).  On the other hand, extensive computer checks\footnote{Many thanks go to Michael Wemyss and Gavin Brown for running these calculations on the database of contraction algebra potentials that they have mined.} go a long way towards confirming the conjecture, at least for $r=1$. In other words, functions in two variables that arise from Abelianizing potentials on the two loop quiver that give rise to finite-dimensional (analytic) Jacobi algebras appear to always have semisimple monodromy.  

For further evidence of the conjecture, for (infinitely many) length 2 flopping curves, and $r=1,2$, the conjecture is proved as a consequence of the total description of the refined invariants found by Okke van Garderen \cite[Thm.A]{VG19}. 

In any case, the more refined invariants $n_{C,r}(z_1,z_2)$ are not necessarily integers.  For instance, let $C$ be a type $(-2,0)$ curve of width $d$.  Then by Example \ref{oneLoopEx}
\[
n_{C,0}(z_1,z_2)=(z_1z_2)^{-1/2}\sum_{i=1}^d z_1^{i/(d+1)}z_2^{(d+1-i)/(d+1)}.
\]
For the rank 1 refined BPS invariants worked out for length two curves by van Garderen, fractional exponents also appear, arising from nontrivial monodromy actions on middle cohomology of (arbitrary genus) complex curves.

\subsection{Cohomological strong rationality conjecture}
\label{CSRconj}
We finish the paper by proposing a mechanism for proving the strong rationality conjecture using a version of the cohomological Hall algebra of \cite{COHA} adapted for analytic Jacobi algebras.  In order to not add greatly to the length of the paper, and since the goal of this section is merely to state a conjecture, we will be a little more sketchy.

There is an obvious similarity between Theorems \ref{ThmA} and \ref{KGVprop}.  Indeed, we could have proved Theorem \ref{ThmA} by instead \textit{defining} the refined Gopakumar--Vafa invariant to be $\GV'_{C,r,0}$.  This similarity is not accidental --- the reason for suspecting that Theorem \ref{ThmA} was true is the easier (given the existing literature) Theorem \ref{KGVprop}, and the Conjecture \ref{SRC} below.  This conjecture implies as a special case that Theorems \ref{ThmA} and \ref{KGVprop} are logically equivalent, once one upgrades (\ref{derIso}) to an isomorphism of (-1)-shifted symplectic stacks.

The derived equivalence (\ref{VdBequiv}) induces an isomorphism of Grothendieck groups
\[
\Phi:\mathbb{Z}^2=\KK(\Coh_C(X))\rightarrow \KK(\hat{A}\rmod_{\nilp})=\mathbb{Z}^2
\]
between the Grothendieck groups of the category of coherent sheaves supported on $C$ and the category of continuous (i.e. nilpotent) modules for the complete noncommutative crepant resolution $\hat{A}$ of Section \ref{CAsec}, which by the same argument as Proposition \ref{ConAl} (i.e. the results of \cite{To18}) is an analytic Jacobi algebra for some quiver $Q$ with convergent formal potential $W$.  The identifications with $\mathbb{Z}^2$ are given by taking rank and Euler characteristic of 1-dimensional sheaves, and dimension vectors of $\hat{A}$-modules.  Then we recover both flavours of cohomological Gopakumar--Vafa invariants as the BPS cohomology $\BPS_{\mathbb{C}\{Q,W\},\Phi(r,d)}$ for $d=0,1$.

Via \cite{To18} for each $\gamma\in\KK(\hat{A}\rmod)$ there is an analytic open neighbourhood of $0_{\gamma}$ in the stack $\mathfrak{M}_{\gamma}(Q)$, for which one may define the function $\Tr(W)$ and take the critical locus, to obtain a stack isomorphic to an analytic open neighbourhood $\Mst$ of the stack of perverse coherent sheaves supported on $C$, inside the stack of all perverse coherent sheaves on $X$.  

Let $\gamma\in\mathbb{N}^2$ be such that $\Phi^{-1}(\gamma)=(a,b)$ with $a\neq 0$.  Then it follows from \cite[Lem.4.9]{To17} that for $U\subset \Msp_{\gamma}(Q)$ an open neighbourhood of $0_{\gamma}$ on which $\Tr(W)$ is holomorphic, $\phim{\Tr(W)}\IC_U$ is supported at $0_{\gamma}$.  In particular, $\BPS_{\mathbb{C}\{Q,W\},\gamma}$ is well defined for all such $\gamma$.  

On the other hand, for $\gamma=\Phi(0,1)$, and $U$ a contractible analytic neighbourhood of $0_{\gamma}\in \Msp_{\gamma}(Q)$ on which $\Tr(W)$ is holomorphic, one can show that $\HO(U, \phim{\Tr(W)}\IC_{U})\cong \HO(\tilde{U},\mathbb{Q})\otimes\LL^{-3/2}$, where $\tilde{U}$ is the neighbourhood of $C$ corresponding to point sheaves for which the semisimplification (as an $A$-module) lies in $U$.  So in particular, $\HO(U, \phim{\Tr(W)}\IC_{U})\cong \HO(C,\mathbb{Q})\otimes\LL^{-3/2}$ is independent of $U$.  The Tate twist is given by the dimension of $\tilde{U}$.  We consider the direct sum across \textit{all} dimension vectors 
\[
\mathcal{H}_{Q,W}=\bigoplus_{\gamma\in\mathbb{N}^{Q_0}}\HO(p_{\gamma}^{-1}(U_{\gamma}),\phim{\Tr(W)}\IC_{p_{\gamma}^{-1}(U_{\gamma})}).
\]
By the same construction as \cite{COHA}, this carries an associative product, and we have embeddings
\[
\BPS_{\mathbb{C}\{Q,W\},\gamma}\otimes \HO(\pt/\mathbb{C}^*)_{\vir}\hookrightarrow \mathcal{H}_{Q,W,\gamma}
\]
for which the induced map
\[
\Sym\left(\bigoplus_{\gamma\in\mathbb{N}^{2}}\BPS_{\mathbb{C}\{Q,W\},\gamma}\otimes \HO(\pt/\mathbb{C}^*)_{\vir}\right)\rightarrow \mathcal{H}_{Q,W}
\]
is the PBW isomorphism from \cite{DaMe15b}.  In particular, letting $u$ denote the (cohomological degree 2) generator in $\HO(\pt/\CC^*,\QQ)\cong\mathbb{Q}[u]$, and letting $1_C$ denote the (cohomological degree 0) generator of $\HO(C,\mathbb{Q})$, we consider the class $u1_C\in \mathcal{H}_{Q,W,\Phi(0,1)}$.  This has cohomological degree zero, due to the Tate twists.
\begin{conjecture}
\label{SRC}
Let $A=\mathbb{C}\{Q,W\}$ be a noncommutative crepant resolution as above.  For $\gamma\in\mathbb{N}^{Q_0}$, the commutator map
\[
[u1_C,-]\colon\mathcal{H}_{Q,W,\gamma}\rightarrow \mathcal{H}_{Q,W,\gamma+\Phi(0,1)}
\]
maps $\BPS_{\mathbb{C}\{Q,W\},\gamma}$ isomorphically to $\BPS_{\mathbb{C}\{Q,W\},\gamma+\Phi(0,1)}$.
\end{conjecture}
The main evidence for the conjecture comes from the related study of 3-folds $\tilde{X}\times\mathbb{C}$, for $\tilde{X}\rightarrow \mathbb{C}^2/\Gamma$ the resolution of a du Val singularity, where the above operator does indeed provide an isomorphism between cohomological BPS invariants\footnote{This will be explained in forthcoming work of Shivang Jindal.}.  Note that an implication of the conjecture is that $\BPS_{\mathbb{C}\{Q,W\},\Phi(r,n)}$ is independent of $n$.  This is a categorified upgrade of the statement that the Donaldson--Thomas invariant for coherent sheaves of rank $r$ and Euler characteristic $d$ are independent of $d$.  In particular, it implies the independence/strong rationality conjectures in \cite{Tod14}, \cite{StablePairs}, \cite{To17}, for flopping curves.

\appendix
\section{Relation between cohomological and motivic DT invariants}
\label{CoMo}
The motivic DT invariants assigned to an (algebraic) Jacobi algebra $A$ in \cite{KS} are equal to the classes $D[\BPS_{A,\gamma}]\in \KK(\MMHS)$ under $\chi_{\MMHS}$, where $D$ is the involution of $\KK(\MMHS)$ induced by Verdier duality, and $\BPS_{A,\gamma}$ is as in Definition \ref{BPSdef}.  This seems to be well-known to the experts (e.g. see \cite[Sec.7.10]{COHA}), but for the convenience of the reader we recall some of the details.

Let $X$ be a smooth complex variety, let $S\subset X$ be a subvariety, and let $f\in \Gamma(X)$ be a regular function.  We refer to \cite{DL01} for the definition of the motivic vanishing cycle 
\[
[\phi_f]\in\mathrm{K}^{\hat{\mu}}_0(\Var/X)[[\AA^1]^{\pm 1/2}]
\]
of $f$.  This is an element of the naive Grothendieck ring of varieties over $X$.  We pick the normalization of this class such that, if $f=0$, $[\phi_f]=[X\xrightarrow{\id_X}X]$.  Given $S\subset X$ we define the map
\begin{align*}
\int_S\colon \mathrm{K}^{\hat{\mu}}_0(\Var/X)[[\AA^1]^{\pm 1/2}]\rightarrow &\mathrm{K}^{\hat{\mu}}_0(\Var/\pt)[[\AA^1]^{\pm 1/2}]\\
[Y\rightarrow X]\mapsto& [Y\times_X S].
\end{align*}
By \cite[Thm.4.2.1]{DL98}, there is an equality 
\[
\chi_{\MMHS}\left(\int_S[\phi_f]\right)=\left[\HO_c\left(X,(\phim{f}\QQQ_X)_S\right)\right],
\]
and so there is an equality
\[
\chi_{\MMHS}\left(\int_S[\phi_f]\cdot [\mathbb{A}^1]^{-\dim(X)/2}\right)=[\HO_c\left(X,(\phim{f}\IC_X)_S\right)].
\]
By Verdier self-duality of $\phim{f}\IC_X$, there is an isomorphism $\HO_c(X,\phim{f}\IC_X)^*\cong \HO(X,\phim{f}\IC_X)$ and so 
\[
D\circ \chi_{\MMHS}\left(\int_X[\phi_f]\cdot[\mathbb{A}^1]^{-\dim(X)/2}\right)=[\HO(X,\phim{f}\IC_X)].
\]
We let $\tilde{\mathrm{K}}_0^{\hat{\mu}}(\Var/\pt)=\mathrm{K}^{\hat{\mu}}_0(\Var/\pt)[[\AA^1]^{\pm 1/2}]\left[ [\Gl_n(\mathbb{C})]^{-1}]\lvert \;n\geq 1\right]$ be the localization formed by formally inverting the classes of the general linear groups.  By the identity
\[
[\Gl_n(\mathbb{C})]=\prod_{i=1}^n([\mathbb{A}^i]-1)
\]
this is the same as the localization with respect to the classes $(1-[\mathbb{A}^n])$ for all $n$.  We let $\KK(\MMHS)_{\leq n}=\chi_{\textrm{wt}}^{-1}(q^n\mathbb{Z}[q^{-1/2}])$.  Since every MMHS has a filtration with subquotients given by pure MMHSs, there is a surjection of Abelian groups $\KK(\MMHS)_{\leq n}\rightarrow \KK(\MMHS)_{\leq n-1}$ and we let $\widehat{\KK}(\MMHS)$ be the inverse limit, with the natural algebra structure.  Let $X$ be a smooth $\Gl_{\gamma}$-equivariant variety, carrying a $\Gl_{\gamma}$-invariant algebraic function $f$, and let $\mathcal{X}=X/\Gl_{\gamma}$ be the quotient stack.  Let $S\subset X$ be a $G$-invariant subvariety, and let $\mathcal{S}\subset \mathcal{X}$ be the induced inclusion of global quotient stacks.  The lowest $l$ for which the weight $l$ piece of 
\[
\HO^i_c(\mathcal{X},\phim{f}\IC_{\mathcal{X}}\lvert_{\mathcal{S}})^*
\]
is nonzero tends to infinity as $i$ tends to infinty, and so the infinite sum
\[
[\HO_c(\mathcal{X},\phim{f}\IC_{\mathcal{X}}\lvert_{\mathcal{S}})^*]=\sum_{i\in \mathbb{Z}}[\HO^i_c(\mathcal{X},\phim{f}\IC_{\mathcal{X}}\lvert_{\mathcal{S}})^*]
\]
converges in $\widehat{\KK}(\MMHS)$.  The class $D\chi_{\MMHS}([\Gl_n(\mathbb{C})])$ is invertible in $\widehat{\KK}(\MMHS)$ and so there is a unique map
\[
J:\tilde{\mathrm{K}}_0^{\hat{\mu}}(\Var/\pt)\rightarrow \widehat{\KK}(\MMHS)
\]
extending $D\circ \chi_{\MMHS}$.  This is a homomorphism of $\lambda$-rings; by construction of the $\lambda$-ring structure, the map $D\circ \chi_{\MMHS}$ is a $\lambda$-ring homomorphism, which extends to a $\lambda$-ring homomorphism
\begin{equation}
\label{localLambda}
\tilde{\mathrm{K}}_0^{\hat{\mu}}(\Var/\pt)\rightarrow\KK(\MMHS)[(\mathbb{L}^n-1)^{-1}\lvert n\geq 1]
\end{equation}
by \cite[Ex.3.5.4(4)]{DM11}.  The inclusion of the right hand side of (\ref{localLambda}) into the ring of formal power series is again a $\lambda$-ring homomorphism by construction.
  
\begin{proposition}
\label{MotComp}
Let $\mathcal{X},\mathcal{S},f$ be as above.  Then 
\begin{equation}
\label{LOP}
J\left(\int_S[\phi_f]\cdot [\mathbb{A}^1]^{-\dim(\mathcal{X})/2}\cdot [\Gl_{\gamma}]^{-1}\right)=[\HO_c(\mathcal{X},(\phim{f}\IC_{\mathcal{X}})_{\mathcal{S}})^*].
\end{equation}
\end{proposition}
\begin{proof}
Set $V_N=\prod_{i\in Q_0}\Hom(\mathbb{C}^N,\mathbb{C}^{\gamma(i)})$ for $N\gg 0$, and let $U_N\subset V_N$ be the subvariety of $Q_0$-tuples of surjective morphisms.  Set
\begin{align*}
\tilde{X}=&(X\times U_N)/\Gl_N\\
\tilde{S}=&S\times_X \tilde{X}.
\end{align*}
Consider the Serre spectral sequence with 
\[
E_2^{p,q}=\HO\left(\prod_{i\in Q_0}\Gr(N,\gamma(i)),\mathbb{Q}\right)\otimes \HO_c(X,\phim{f}\QQQ_{X}\lvert_S).
\]
Since this spectral sequence converges to $\HO_c(\tilde{X},\phim{f}\underline{\mathbb{Q}}_{\tilde{X}}\lvert_{\tilde{S}})$ we deduce that
\begin{align}
\label{SpecSeqK}
&[\HO_c(\tilde{X},\phim{f}\underline{\mathbb{Q}}_{\tilde{X}}\lvert_{\tilde{S}})\otimes\LL^{-N\lvert\gamma\lvert-\dim(\mathcal{X})/2}]=\\&\nonumber \prod_i[\HO(\Gr(N,\gamma(i)),\mathbb{Q})]\cdot[\HO_c(X,\phim{f}\underline{\mathbb{Q}}_X\lvert_{S}\otimes\LL^{-N\lvert\gamma\lvert-\dim(\mathcal{X})/2}].
\end{align}
Letting $N$ tend to infinity, the left hand side of (\ref{SpecSeqK}) converges to the right hand side of (\ref{LOP}).  The group $\Gl_{\gamma}$ is special, meaning that principal $\Gl_{\gamma}$ bundles are Zariski locally trivial.  It follows from the cut and paste relations in $\mathrm{K}_0^{\hat{\mu}}(\Var/\mathbb{C})$ that 
\[
[\Gr(N,\gamma_i)]\cdot[\Gl_{\gamma}]=[U_N].
\]
Applying $J$, we deduce
\[
D[\HO(\Gr(N,\gamma_i),\mathbb{Q})\otimes\LL^{-N\lvert\gamma\lvert}]\cdot D[\HO_c(\Gl_{\gamma},\mathbb{Q})]=D[\HO_c(U_N,\mathbb{Q})\otimes\LL^{-N\lvert\gamma\lvert}].
\]

The right hand side tends to $1$ as we let $N$ tend to infinity, and the result follows.
\end{proof}
We next consider the motivic DT invariants of a quiver $Q$ with potential $W$.  As above, we should choose a Serre subcategory $\mathcal{S}$ of the category of representations of $\mathbb{C}(Q,W)$, and we make our standard assumption that $\mathcal{S}$ is the maximal permissible Serre subcategory, containing all modules $M$ such that on all subquotients of $M$, the value of $\Tr(W)$ is zero.  

Motivic DT invariants are defined in terms of plethystic exponentials, which we very briefly recall.  Given $\alpha$ an element in a $\lambda$-ring, we define $\Exp(\alpha)=\sum_{n\geq 0}\sigma^n(\alpha)$, when this sum makes sense.  So for instance if $B=\KK(\mathcal{A_{\mathbb{N}}})$ is the Grothendieck ring of $\mathbb{N}$-graded objects of an Abelian tensor category, completed so that $[V]=\sum_{n\in \mathbb{N}}[V_i]$ for $V_i$ the $i$th graded piece of $V_i$, and $V\in \mathcal{A}_{\mathbb{N}}$ is concentrated in strictly positive degrees, then
\begin{align*}
\Exp([V])=&\sum_{n\geq 0}[\Sym^n(V)]\\
=&[\Sym(V)].
\end{align*}
Let $(Q,W)$ be an algebraic quiver with potential.  Let 
\begin{align*}
\chi\colon\mathbb{Z}^{Q_0}\times \mathbb{Z}^{Q^0}\rightarrow &\mathbb{Z}\\
(\gamma,\gamma')\mapsto& \sum_{i\in Q_0} \gamma_i\gamma'_i-\sum_{a\in Q_1} \gamma_{s(a)}\gamma'_{t(a)}
\end{align*}
denote the Euler form.  There is an equality
\[
\dim(\mathbb{A}_{Q,\gamma}/\Gl_{\gamma})=-\chi(\gamma,\gamma).
\]
Let $g_{\gamma}\in\Gamma(\mathbb{A}_{Q,\gamma})$ be the function induced by $\Tr(W)$.  Then the motivic DT invariants are defined by the equation
\begin{align*}
\sum_{i\in\mathbb{N}^{Q_0}}\int_{\mathcal{S}}[\phi_{g_{\gamma}}]\cdot[\mathbb{A}^1]^{\chi(\gamma,\gamma)/2}\cdot[\Gl_{\gamma}]^{-1}T^{\gamma}=&\Exp\left(\sum_{0\neq i\in \mathbb{N}^{Q_0}}\Omega_{Q,W,\gamma}[\mathbb{C}^*]^{-1}\cdot[\mathbb{A}^1]^{1/2}T^{\gamma}\right).
\end{align*}
We deduce from Proposition \ref{MotComp} that
\[
J\left(\int_{\mathcal{S}}[\phi_{g_{\gamma}}]\cdot[\mathbb{A}^1]^{\chi(\gamma,\gamma)/2}\cdot[\Gl_{\gamma}]^{-1}\right)=\left[\HO\left(\Mst_{\gamma}(Q),(\phim{\Tr(W)}\IC_{\Mst_{\gamma}(Q)})_{\mathcal{S}}\right)\right].
\]
Taking the hypercohomology of (\ref{Sint}), we deduce that
\begin{align*}
&\sum_{\gamma\in\mathbb{N}^{Q_0}}J\left(\int_{\mathcal{S}}[\phi_{g_{\gamma}}]\cdot[\mathbb{A}^1]^{\chi(\gamma,\gamma)/2}\cdot[\Gl_{\gamma}]^{-1}\right)T^{\gamma}=\\
&\Exp\left(\sum_{0\neq \gamma\in\mathbb{N}^{Q_0}}[\BPS_{\mathbb{C}(Q,W),\gamma}]\cdot [\HO(\pt/\mathbb{C}^*)_{\vir}]T^{\gamma}\right)
\end{align*}
and we deduce the following
\begin{proposition}
There is an equality in $\widehat{\KK}(\MMHS)$
\[
J(\Omega_{Q,W,\gamma})=[\BPS_{Q,W,\gamma}].
\]
In particular, the left hand side lies in the image of the map $\KK(\MMHS)\hookrightarrow \widehat{\KK}(\MMHS)$.
\end{proposition}

\bibliographystyle{amsplain}
\bibliography{SPP}

\vfill

\textsc{\small B. Davison: School of Mathematics, the University of Edinburgh}\\
\textit{\small E-mail address:} \texttt{\small ben.davison@ed.ac.uk}\\
\\

\end{document}